\documentclass[a4paper,twoside]{article}

\usepackage{amsmath,amsthm,amsfonts,latexsym,amscd,amssymb,enumerate}
\usepackage[backref]{hyperref}
\usepackage{color}

\newcommand{\specialvalue}{\frac{1}{2}}
\swapnumbers
\theoremstyle{plain}
\newtheorem{theorem}{Theorem}[section]
\newtheorem{lemma}[theorem]{Lemma}
\newtheorem{corollary}[theorem]{Corollary}
\newtheorem{proposition}[theorem]{Proposition}

\theoremstyle{definition}
\newtheorem{definition}[theorem]{Definition}

\theoremstyle{remark}
\newtheorem{remark}[theorem]{Remark}

\newcommand{\reals}{\mathbb{R}}

\newcommand{\naturals}{\mathbb{N}}
\newcommand{\integers}{\mathbb{Z}}
\newcommand{\rationals}{\mathbb{Q}}

\DeclareMathOperator{\id}{id}

\newcommand{\abs}[1]{\left\lvert#1\right\rvert}

\newcommand{\tensor}{\otimes}

\newcommand{\subgroup}{\leq}

\newcommand{\semiProd}{\rtimes}
\newcommand{\semiprod}{\semiProd}

\DeclareMathOperator{\im}{im}      
\DeclareMathOperator{\Span}{span}

\DeclareMathOperator{\tr}{tr}

\newcommand{\forget}[1]{}

\newcommand{\innerprod}[1]{\langle #1 \rangle}

{\catcode`@=11\global\let\c@equation=\c@theorem}

\allowdisplaybreaks[2]

\newcommand{\G}{\Gamma}
\DeclareMathOperator{\Cay}{Cay} 

\begin{document}

\pagestyle{myheadings}
\markboth{Mika{\"e}l Pichot, Thomas Schick, Andrzej Zuk}{Closed
  manifolds with transcendental $L^2$-Betti numbers} 


\title{Closed manifolds with transcendental $L^2$-Betti
  numbers} 

\author{ Mika\"el Pichot\thanks{Mika\"el Pichot was supported by JSPS and the WPI Initiative, MEXT, Japan
\protect\\
\protect\href{mailto:pichot@math.mcgill.ca}{e-mail:
  pichot@math.mcgill.ca}} \\
McGill University\\
Montreal, Canada 
\and Thomas Schick\thanks{Thomas Schick was partially supported by the Courant Research Center
    ``Higher order structures in Mathematics'' within the German 
initiative of
    excellence\protect\\
\protect\href{mailto:thomas.schick@math.uni-goettingen.de}{e-mail:
  thomas.schick@math.uni-goettingen.de}
}\\
Georg-August-Universit\"at G{\"o}ttingen\\
Germany
\and Andrzej Zuk\thanks{Andrzej Zuk was supported by the Humboldt
  foundation\protect\\ email: zuk@math.jussieu.fr}\\ Institute
Math{\'e}matiques de Jussieu\\
Paris, France\thanks{All authors were partially supported by HIM, Bonn. The
  main part of work was carried out during the HIM trimester program ``Rigidity''.}
}
\maketitle

\begin{abstract}
    In this paper, we show how to construct examples of closed manifolds with
    explicitly computed 
  irrational, even transcendental $L^2$ Betti numbers, defined via the
  universal covering.

  We show that every non-negative real number shows up as an $L^2$-Betti
  number of some covering of a compact manifold, and that many 
  computable real numbers appear as an $L^2$-Betti number of a universal
  covering of a compact manifold (with a precise meaning of computable given below).

  In algebraic terms, for many given computable real numbers (in
  particular for many 
  transcendental numbers) we show how to construct a finitely presented group
  and an element in
  the integral group ring such that the $L^2$-dimension of the kernel is the
  given number.

  We follow the method pioneered by Austin
  \cite{austin09:_count_to_quest_of_atiyah} but refine it to get 
   explicit calculations which make the above statements possible.

\end{abstract}

\noindent
2010 Mathematics Subject Classification 05C25 (primary), 03D40, 20F65, 55N25 (secondary).

\section{Introduction}
In 1974, Atiyah defined $L^2$-Betti numbers for covering spaces of closed
manifolds \cite{MR0420729}.  A priori these Betti numbers are real and   
Atiyah asked at the end of his paper to find examples where they are irrational.
The question remained open and the fact that these $L^2$-Betti numbers may always be rational, and even integral for torsion free groups,  has become known 
as the ``Atiyah conjecture''. Under conditions on the torsion in the group, 
more refined conjectures have been formulated and popularized as the ``strong Atiyah
conjectures''  \cite[Chapter 10]{MR1926649}, \cite[Definition 1.1]{MR1990479}, which are satisfied for many groups. 

Let us observe that the discussion is concerned with two slightly different cases:
\begin{itemize}
\item Atiyah from the very beginning studied arbitrary normal coverings of a
  compact manifold $M$. The resulting values for the $L^2$-Betti 
  numbers may be very different depending on which covering of $M$ they are associated with.
\item The most important special case, often exclusively considered in later
  work, uses the universal covering of the manifold $M$. This way, one
  defines invariants depending only on $M$: these are the invariants usually referred to as \emph{the} $L^2$-Betti numbers of $M$.
\end{itemize}

The $L^2$-Betti numbers are homotopy invariants of the underlying manifold
$M$. It follows from this that, when considering only the universal covering,
i.e.~\emph{the} $L^2$-Betti numbers, there is in total only a countable set of possible
values.

However, a given space can have uncountably many different normal coverings
(corresponding to the normal subgroups of the fundamental group) so that the
set of possible $L^2$-Betti numbers of normal coverings of compact manifolds
a priori could well be uncountable.

In a recent paper \cite{austin09:_count_to_quest_of_atiyah}, Tim Austin 
showed that the set of $L^2$-Betti numbers associated to all possible normal coverings of
compact manifolds is uncountable, and in particular contains irrational (and even transcendental) values. 

In the present paper, we show how to construct examples of closed manifolds with
explicitly given irrational (and transcendental) $L^2$-Betti numbers
for their universal coverings. As explained below, we follow closely the techniques 
developed by Austin in \cite{austin09:_count_to_quest_of_atiyah}, with refinements which
allow us to make explicit dimension calculations. {Explicit calculations (and
to some extend the basis of all these developments) have been carried out
previously in \cite{MR1934693, MR1866850, MR1797748}, which already lead to unexpected
values of $L^2$-Betti numbers, not however to any which one could prove to be irrational.}

The problem at hand
has a well known purely algebraic reformulation.
The aim is to produce a finitely presented group $G$
and an element $Q$ in the group algebra  $\mathbb{Z}[G]$ such that
$\dim_{G} \ker(Q)$, the von Neumann
dimension of the kernel of this operator acting on $l^2(G)$ is
irrational. Then there is a standard construction to obtain 
a  closed $7$-dimensional manifold $M$ with the fundamental group isomorphic to
$G$ 
and whose third $L^2$-Betti number (computed using the universal covering) is
equal to the $\dim_G  \ker(Q)$, see 
\cite[Lemma 10.5]{MR1926649} and \cite[Proposition 6 and Theorem 7]{MR1797748}.

If, instead of starting with a finitely presented group one only starts with
a finitely generated group $G$, the standard construction will result in a
manifold $M$ with normal covering $\overline M$ (which is not necessarily the
universal covering) such that the third $L^2$-Betti number for this covering
is equal to $\dim_G \ker(Q)$.

Actually, we construct a group $G$ which is not finitely presented but admits a 
recursive presentation and thus embeds into a finitely presented group
$H$ by Higman's theorem \cite{MR0130286}. For
a suitable element $Q \in \mathbb{Q}[G]$ we prove that $\dim_G \ker(Q)$
is transcendental. Clearing denominators, we can achieve that $Q \in
\integers[G]$ without changing its kernel. Finally, it is a standard fact that
the dimension of the kernel does not change if 
we let $Q$ act on $l^2(H)$, compare e.g.~\cite[Proposition
3.1]{MR1828605}.

The group $G$ will be of the form
$$ ( \mathbb{Z}_2^{\oplus {\G}}  / V )  \rtimes \G   $$
where $V$ is a suitable $\G$-invariant subspace of
$\mathbb{Z}_2^{\oplus {\G}}  $. For 
$\G$, we will choose either the free group on two generators $F_2$ (as in
\cite{austin09:_count_to_quest_of_atiyah}) or $\integers\wr\integers$.

The main result of \cite{austin09:_count_to_quest_of_atiyah} is to construct
an uncountable family of groups $G_i$ of the form above and operators $Q_i \in \mathbb{Q}[G_i]$
such that the numbers $\dim_{G_i} (\ker(Q_i))$ are all mutually different. It seems
hard to prove that among those groups for which 
$\dim_{G_i} \ker(Q_i)$ is irrational are recursively presented groups, as their
existence is only inferred from a counting argument.

In contrast, in the paper at hand we consider different operators $Q$
for which we manage to explicitly
compute $\dim_G\ker(Q)$. Along the way, we explicitly produce a recursively
presented group   $G$
for which  $\dim_G  \ker(Q)$ is transcendental.

Namely, for any set of natural numbers $I=\{0, n_k \}\subset \naturals$ (listed in
increasing order $0<n_1<n_2<\dots$) we construct a group $G_I$ as above whose presentation is
determined by the set $I$ together with $Q_I\in \rationals[G_I]$ such that
$$ \dim_{G_I} \ker(Q_I)=\beta_1 + \beta_2  \sum_{k=1}^\infty   2^{-d
    n_k +k}   $$ 
where $\beta_1$ and $\beta_2$ are some explicit rational numbers and $d$ is a
natural number.

We prove that $G_I$ has a recursive presentation (and therefore embeds into a
finitely presented group) if (and only if) $I$ is recursively enumerable.
It is now immediate to choose a recursively enumerable set $I$ which leads to an irrational
or even transcendental $L^2$-dimension, e.g.~by asking it to satisfy the Liouville
condition. Recall that a real number $x$ is a Liouville number if for any positive integer $n$, there exist integers $p$ and $q$ with $q > 1$ and such that
$ 0 <  \left   |   x - \frac{p}{q}        \right  |  <  \frac{1}{q^n} $.
Liouville \cite{liouville44:_sur_des_class_tres_etend} showed that such
numbers are transcendental.

Let us also stress the fact that we obtain these $L^2$-Betti numbers for
solvable groups (this is the reason why we use
the group $\G=\integers\wr\integers$), answering a question of
\cite{austin09:_count_to_quest_of_atiyah}. Note that for
torsion-free solvable groups the Atiyah conjecture is known \cite[Theorem 1.3]{MR1242889}.

Using the explicit form of these $L^2$-dimensions of kernels for the operators we
obtain, we can construct out of these for each real number
$r\ge 0$ a group $G_r$ (in general not recursively presented) and $A_r\in
M_n(\integers[G_r])$ with $\dim(\ker(A_r))=r$. This relies on explicit
knowledge of how the kernel looks like under the operations we employ.

We will discuss possible extensions of the result which can be
obtained with the same method. In particular, with suitable modifications and
additional effort one could produce many examples of $A\in M_n(\integers[G])$
as above with explicit knowledge of the full spectral measure (as in
\cite{MR1934693,MR1866850}). This spectral
measure would be atomic and many of the $L^2$-dimensions of the eigenspaces
would be transcendental.

We also discuss more about the question which $L^2$-Betti numbers can (by
modifications of the construction) be obtained using finitely presented groups.

\bigskip

  Lukasz Grabowski
  \cite{grabowski10:_turin_machin_dynam_system_and_atiyah_probl} has independently and simultaneously, using an approach
  which implements Turing machines directly in the integral group ring of a
  suitable recursively presented group, arrived at results similar to
  ours. Using yet another strategy, Franz Lehner and Stephan Wagner
  \cite{lehner:_free_lampl_group_and_quest_of_atiyah} manage in a clever way to
  calculate explicitly the $L^2$-dimension of the kernel of the usual graph Laplacian (as in \cite{MR1934693}), but for wreath
  products of finite groups with non-abelian free groups, and find
  non-rational algebraic numbers among the values---giving yet another
  interesting example.
\smallskip

\textbf{Acknowledgements}. The results of this paper were obtained during the Trimester Program on
Rigidity at the Hausdorff Institute of Mathematics, Bonn. The authors thank
HIM for the stimulating atmosphere and the generous support to be able to
participate in this program. The authors also thank Tim Austin for fruitful discussions.

\section{Preliminary remarks}

We closely follow the notations of
\cite{austin09:_count_to_quest_of_atiyah}. This will hopefully help a reader
interested in reading concurrently both papers. 

We consider groups of the form 
$(\mathbb{Z}_2^{\oplus \G} / V)  \rtimes_\alpha \G$
where $V$ is some left translation invariant subgroup of $\mathbb{Z}_2^{\oplus
  \G}$ and the action $\alpha$ of $\G$ is by translations on the
left. Usually, we will omit $\alpha$ in the notation.

For the first few sections, we will only assume that $\G$ is generated by two
elements $s_1,s_2$ and that it satisfies the zero divisor
conjecture over $\mathbb{Z}_2$. The latter means that the group ring
$\integers_2[\G]$ contains no non-trivial zero divisors. It implies in
particular that the group is torsion free, so $s_1,s_2$ are of infinite order.
 Later, concrete computations  will be
carried out notably in the case of the free group $\G=F_2$  and of the
wreath product $\G=\mathbb{Z}\wr \mathbb{Z}$ with natural generating sets,
which of course satisfy all these conditions. 

We denote by $\Cay(\G,S)$ the right Cayley graph of $\G$ with respect to $S=\{s_1^{\pm 1},s_2^{\pm 1}\}$. By a path $P$ in $\Cay(\G,S)$ we mean
 a subset $\{g_1,g_2,g_3,\ldots ,g_\ell\}\subset \G$ of mutually distinct
 consecutive elements of $\Cay(\G,S)$.

We recall that the Pontryagin dual of $\mathbb{Z}_2^{\oplus \G}$  is
isomorphic to the infinite product $\mathbb{Z}_2^\G$. Sometimes we will
identify an element $\chi\in \integers_2^\G$ with the subset
$\chi^{-1}(1)\subset \G$. We denote by 
\[
V^\perp =\{\chi \in \mathbb{Z}_2^\G \mid \langle \chi\mid v\rangle =0~~\forall v\in V\}
\]
the dual of $\mathbb{Z}_2^{\oplus \G} / V$ and recall that the Fourier transform 
\[
\ell^2(\mathbb{Z}_2^{\oplus \G}/V)\simeq L^2(V^\perp,m_{V^\perp})
\]
where $m_{V^\perp}$ is the Haar measure on $V^\perp$, induces a spatial isomorphism 
between the group von Neumann algebra $L((\mathbb{Z}_2^{\oplus \G} / V)  \rtimes
\G)$ and the cross-product von Neumann algebra $L^\infty(V^\perp)\rtimes_{\hat
  \alpha} \G$. The dual action $\hat \alpha$ is on $V^\perp$  defined by
$\langle \hat \alpha_g\chi\mid u\rangle=\langle  \chi\mid
\alpha_{g^{-1}}u\rangle$ for $\chi\in V^\perp$ and $u\in \mathbb{Z}_2^{\oplus
  \G} / V$. Note that, if we think of $\chi\in \integers_2^\G$ as a subset of
$\G$, then $\hat\alpha_g\chi$ corresponds to the subset $g\chi^{-1}(1)$,
i.e.~is obtained by a left translation by $g$. 

 For simplicity we sometimes denote $s(\chi):=\hat \alpha_s\chi$. 
For $F\in  L^\infty(V^\perp)$ we denote by $M_F\in
L^\infty(V^\perp)\rtimes_{\hat \alpha} \G$ the twisted pointwise
multiplication, defined by 
\[
M_F f(\chi,g)=F(\hat \alpha_g(\chi))\cdot f(\chi,g)
\] 
and by $T_{s}$ the translation operator given by $T_sf(\chi,g)=f(\chi,s^{-1}g)$.
Checking the definitions, we observe the covariant relation 
\begin{equation}\label{eq:covariance}
  T_{s^{-1}} M_F T_s=M_{F\circ\hat\alpha_s}.
\end{equation}

\section{The revised operators}

We want to construct certain operators in the rational group ring $ \mathbb{Q}
[\mathbb{Z}_2^{\oplus \G} / V  \rtimes \G]$, viewed as acting on
$L^2(V^\perp,m_{V^\perp})\bar \otimes \ell^2(\G)$.
They will be taken to be of the form
\begin{equation}\label{eq:operator}
A=\sum_{s\in S} T_{s^{-1}}(M_{{F_s}_{|V^\perp}} +
M_{{F_{{s^{-1}}}\circ\hat\alpha_{s^{-1}}}_{|V^\perp}}) ,
\end{equation}
where $F_s : \mathbb{Z}_2^\G \to \mathbb{Q}$ will depend only on finitely
many coordinates around the origin $e$. The operator $A$ is self-adjoint
as shown in \cite[Lemma 3.1]{austin09:_count_to_quest_of_atiyah}.

The essential difference with the operators of
\cite{austin09:_count_to_quest_of_atiyah} is that the function $F_s$ will
recognize a very specific family of paths that we call ``hooks'' and which
substitute the paths ``with no small horizontal doglegs''  of \cite[Definition
3.2]{austin09:_count_to_quest_of_atiyah}. This  one ingredient simplifies
several computations and is what allows us to calculate the von
Neumann dimensions exactly.

\begin{definition} A  path $P$ (finite or not)
in $\Cay(\G,S)$ is called a \emph{hook} if it has the form
\begin{equation}\label{eq:def_of_hook}
P=\left \{
gs_2^{-n}, \ldots, gs_2^{-1}, g , g s_1, g s_1 s_2^{-1}, \ldots , g s_1 s_2^{-m}
\right \}
\end{equation}
for some $g \in \G$ and $n,m\in \{1,\dots,\infty\}$. If $n<\infty$ then $gs_2^{-n}$ is called
  the \emph{left endpoint} of $P$. We call $n$ the \emph{length of the left
    leg} and $m$ the \emph{length of the right leg}.

We call a path $P$ (finite or not) a \emph{vertical segment} if 
\begin{equation*}
  P=\left\{ gs_2^{-n},\ldots,g,\ldots,gs_2^m\right\} \qquad\text{for some }n,m\in\integers.
\end{equation*}

If $h,hs_2\in P$, but $hs_2^{-1}\notin P$ we call $h$ a \emph{lower endpoint}
of the hook or vertical segment $P$. If $P$ is a hook and $h$ is additionally
the left endpoint, then $h$ is called \emph{left lower endpoint}.
\end{definition}

We denote by $B(g,k)$ the ball of radius $k$ around $g\in \G$ in $\Cay(\G,S)$.

\begin{definition}
Let  $\chi \in \mathbb{Z}_2^\G$. We say that $\chi$ is \emph{$1$-good} if for
some hook $P$ in $\Cay(\G,S)$ containing $e$, its restriction $\chi_{|B(e,1)}$
to $B(e,1)$ equals $1$ on $P$ and $0$ outside. We say that $\chi$ is
\emph{locally
  good} if $\chi$ is $1$-good and $s(\chi)$
is $1$-good for every $s\in \chi^{-1}(1)\cap B(e,1)$. We say that $\chi$ is
\emph{interior good} if $\chi$ is locally good and
$\abs{\chi^{-1}(1)\cap B(e,1)}=3$. We say that $\chi$ is a
\emph{good end} if $\chi$ is locally good and $e$ is a lower endpoint of the
hook $P$ above. This happens exactly if $\abs{\chi^{-1}(1)\cap B(e,1)}=2$. 
\end{definition}

We now introduce  $F_s\colon \mathbb{Z}_2^\G \to
\mathbb{Q}$. If $\chi \in \mathbb{Z}_2^\G$  is interior
good then
\begin{itemize}
\item[(i)] $F_s(\chi) := 1$ if $s^{-1}(\chi)$ is interior  good;
\item[(e)] $F_s(\chi) :=2$ if 
  $s^{-1}(\chi)$ is a good end;
\item[(b)] $F_s(\chi):=\specialvalue$ if $s^{-1}(\chi)$ is $1$-good, but not locally good.
\end{itemize}
Define $F_s(\chi):=0$ otherwise. Note that this happens if $\chi$ is not interior
good or if $\chi$ is, but $s^{-1}\chi$ is not
$1$-good.

In the definition of the operator $A$, both $F_s(\chi)$ and
$F_{s^{-1}}\hat\alpha_{s^{-1}}(\chi)$ appear. For further reference, we compile
  a little table showing the values of these two functions for the different
  possibilities. The columns give the different properties of $\chi$, the rows
  those of $s^{-1}\chi=\hat\alpha_{s^{-1}}\chi$. The first value in each entry is
  $F_s(\chi)$, the second one $F_{s^{-1}}(s^{-1}\chi)$.

\noindent
  \begin{tabular}{|l|l|l|p{2cm}|l|}\hline
    $s^{-1}\chi$ $\backslash$ $\chi$ & int.~good & good end & $1$-good, not
    loc.~good & not $1$-good\\ \hline
  int.~good & $1;\; 1$ & $0;\; 2$ & $0 ; \; \specialvalue$ & $0;\;0$\\\hline
  good end & $2;\;0$& $0;\;0$& $0;\;0$ & $0;\;0$\\ \hline
  $1$-good, not loc. good & $\specialvalue;\;0$ & $0;\;0$ & $0;\;0$ &$0;\;0$\\ \hline
  not $1$-good & $0;\;0$ & $0;\;0$& $0;\;0$ & $0;\;0$\\ \hline
  \end{tabular}

\smallskip

This follows by inspecting the definition of $F_s$. Note that $F_s(\chi)$
depends on $\chi$ and $s^{-1}\chi$, whereas
$F_{s^{-1}}\hat\alpha_{s^{-1}}(\chi)=F_{s^{-1}}(s^{-1}\chi)$ depends on
$s^{-1}\chi$ and $(s^{-1})^{-1}s^{-1}\chi=\chi$, which of course also explains
why the second matrix we obtain is the transpose of the first. 

Finally note that $A$ is a sum of operators of the form $T_{s^{-1}} M_{G_{s}}$ where
$G_{s}:=F_{s}+F_{s^{-1}}\hat\alpha_{s^{-1}}$
itself is a linear combination of characteristic functions, and $G_{s}$ depends
only on the $3$-neighborhood of $e$. For later
reference and convenience we list the relevant values of $G_s(\chi)$
next. Note that $G_s(\chi)$ depends on $\chi$ and $s^{-1}\chi$; we will give a
 description now.
\begin{proposition}\label{prop:precise_A}
   \begin{enumerate}
   \item $G_s(\chi)=0$ if $\chi$ is not a $1$-good because then neither $\chi$
     nor $s^{-1}\chi$ is interior good.
   \item (Case where $\chi$ is an end): Assume next that $\chi$ is $1$-good
     and $\chi^{-1}(1) \cap B(e,1)  =\{e,s_2\}$. Then $G_s(\chi)=0$ if $s\ne s_2$ because
     then $s^{-1}\chi$ is not $1$-good. Moreover, $G_{s_2}(\chi)=2$ if
     $s_2^{-1}\chi$ is interior good (i.e.~the path extends two more steps)
     and $G_{s_2}(\chi)=0$ otherwise. 
   \item Now assume that $\chi$ is $1$-good and
     $\chi^{-1}(1)=\{e,s_2^{-1},t\}$. For $s\ne s_2^{-1},t$, $G_s(\chi)=0$
     because then $s^{-1}\chi$ is not $1$-good. Moreover,
     \begin{enumerate}
     \item (case where in one direction the path goes bad): if for $s\in
       \{s_2^{-1},t\}$ $s^{-1}\chi$ is not $1$-good (i.e.~the path doesn't
       extend in this direction) then $G_s(\chi)=0$. Write
       $\{s_2^{-1},t\}=\{s,s'\}$, then $G_{s'}(\chi)=0$ if ${s'}^{-1}\chi$ is
       not interior good, and $G_{s'}(\chi)=\specialvalue$ otherwise.
     \item (case where in one direction, necessarily $s_2^{-1}$, the path
       ends). If $s_2\chi$ is a good end then $G_{s_2^{-1}}(\chi)=2$ if $\chi$
       is interior good and $G_{s_2^{-1}}(\chi)=0$ if $\chi$ is not interior
       good (the latter situation we've just discussed).
     \item Assume now that $\chi$ is interior good, $s\in\{s_2^{-1},t\}$ and
       $s^{-1}\chi$ is interior good (i.e.~the hook extends in two directions
       through $e$, and in direction $s$ even two steps). Then $G_s(\chi)=2$.
     \end{enumerate}
   \end{enumerate}
 \end{proposition}
\begin{remark}\label{rem:colloquial_A}
\begin{enumerate}
\item   In other words: we only ``move along the path'', with weight $2$ if one is
  in an interior situation or arrives at or from a good end point (with some
  extension of the path in all directions). We use weight $\specialvalue$ if we move to
  or from a point which is next to a bad point (again the path has to extend a
  bit in the other direction).
\item Our definition of $F_s$ involves $1$-neighborhoods rather than
  $10$-neighborhoods. This will make calculations later easier, in particular
  if 
  $\G$ is not the free group. In the framework of
  \cite{austin09:_count_to_quest_of_atiyah} 
  one can economize and can reduce the size of the neighborhoods, albeit not 
  to $1$.   
\item We emphasize that our definition of $F_s$ makes the operators $A$ follow
  the hook itself, rather than its $1$-neighborhoods (this convenient
  simplification will be made precise in Section \ref{sec:decomposition_of_V}). 
\end{enumerate}
\end{remark}

\section{Decomposition of $V^\perp$ into invariant subsets}
\label{sec:decomposition_of_V}

This section follows pretty much   \cite[Section
3.2]{austin09:_count_to_quest_of_atiyah}, with slight modifications that we
indicate now.

\begin{definition}
Given $\chi\in\integers_2^\G$, a ball $B(g,1)$ is called a \emph{good
  neighborhood} of $\chi$, if $g^{-1}(\chi)\cap B(e,1)=\{e,s_2\}$. $B(g,1)$ is
called a \emph{bad 
  neighborhood} if $g^{-1}(\chi)$ is not $1$-good. 
\end{definition}

Having this definition, and since our notations essentially coincide, we can obtain a
partition of $\mathbb{Z}_2^\G$ by simply copying that of
\cite[Section 3]{austin09:_count_to_quest_of_atiyah}. Namely, we obtain first
a disjoint Borel partition
\[
\mathbb{Z}_2^\G=C_0\cup C_{1,1}\cup C_{1,2} \cup C_{2,2} \cup C_{1,\infty}
\cup C_{2,\infty} \cup C_{\infty,\infty}.
\]
Here $C_0$ is the set of $\chi$ such that  $F_s(\chi)$ and
$F_{s^{-1}}(s^{-1}\chi)$ are both zero for all generators $s$. If $\chi\notin
C_0$ then in particular $\chi$ is $1$-good, i.e.~$\chi^{-1}(1)\cap B(e,1)$
contains a piece of a path containing $e$.

To which of the other sets $C_{i,j}$ a $\chi\notin C_0$ belongs is now
determined according to the fate of two walkers
starting at the origin and moving in opposite directions along this path
starting at $e$. Indeed, for this description we identify $\chi$ with the
subset $\chi^{-1}(1)$ of $\G$. 
Each walker will have as path a (possibly infinite)
hook or vertical segment $R'$, starting at $e$.  

We have three possible disjoint ending scenarios $i,j\in \{1,2,\infty\}$ for
each walker:
\begin{itemize}
\item[($\infty$)] the walker never reaches a good or a bad neighborhood, and
  continues his path forever; 
\item[(1)] the walker reaches a good neighborhood and ends up at a lower end point
  of a hook;
\item[(2)] the walker reaches a bad neighborhood and stops walking.
  In this
   case we let $P'\subset R'$ be the path that the given walker follows up to
   distance $1$ of his stopping point. 
\end{itemize}
 Furthermore, in case (1) we set $P':=R'$ to be the path followed by our given
 walker.
Finally, we set $R(\chi)$ to be the union of the two hooks $R'$ of the two
walkers, and define $P(\chi)$ similarly.
Note that these are hooks or vertical segments and that
the intersection of $\chi$ with the $1$-neighborhood of each endpoint of $P$
(if it exists) determines the ending scenario at that endpoint. Finally, let
$\psi(\chi)$ be the restriction of the function $\chi$ to the $1$-neighborhood
of $R$.   It takes the value $1$ on $R$ and $0$ on all 
points outside $R$ except possible on the $1$-neighborhood of $R\setminus P$,
the set of ``bad'' endpoints (if they exist).

If
$i,j<\infty$, $R$ and $P$ are finite.
Note that we have no way to order the walkers a priori: only the unordered
tuple of ending scenarios is significant.

Next, we make a further refinement and decompose each set $C_{i,j}$ for $i\le j<\infty$ according to
abstract triples $(P,R,\psi)$ which occur in the above discussion. 
Let $\Omega_{i,j}$ be the set of all such triples and set
\[
C_{(P,R,\psi)} =\{\chi\in \mathbb{Z}_2^\G \mid R(\chi)=R,\;P(\chi)=P,\;\psi(\chi)=\psi\}.
\]

We obtain:

\begin{lemma} The following is a Borel partition of  $\mathbb{Z}_2^\G$:
\[
 \mathbb{Z}_2^\G=C_0\cup \left ( \bigcup_{i,j\in {1,2}\ i\leq
     j}\bigcup_{(P,R,\psi)\in \Omega_{i,j}}  C_{(P,R,\psi)} \right)\cup
 C_{1,\infty} \cup C_{2,\infty} \cup C_{\infty,\infty}.
\]
\end{lemma}

By intersection with $V^\perp$ this leads to a Borel partition of $V^\perp$
and therefore to an orthogonal
decomposition of $L^2(V^\perp)$. Depending on $V^\perp$, several summands might
vanish.

\bigskip

Later, the pile-up of eigenspaces is organised according to the following equivalence
relations on triples $(P,R,\psi)$: 

\begin{definition}
Two triples $(P_1,R_1,\psi_1)$ and $(P_2,R_2,\psi_2)$ are said to be
\emph{translation equivalent} if there exists a $g\in \G$ such that
$P_2=gP_1$, $R_2=gR_1$ and $\psi_2(gh)=\psi_1(h)$ for all $h\in
U(R_1,1)$, where for $k\in\naturals$ and $X\subset \G$ we let $U(X,k)$ denote
the $k$-neighborhood of $X$ in the Cayley graph of $\G$.
\end{definition}

Equivalence classes in $\Omega_{i,j}$ are finite (since $R$ is finite and contains $e$) and
we denote them by $\mathcal{C}\in \Omega_{i,j}/\sim$. Moreover, note that if
$e\in P$ then the set of $g\in\G$ which translates $(P,R,\psi)$ to a
translation equivalent pair are exactly the $g\in P$ with $g^{-1}\in P$.

\section{Unitary equivalence}
\label{sec:unitary_equiv}

We obtain the decomposition 
\begin{equation}\label{decomposition}
L^2(V^\perp) \bar\otimes\ell^2(\G) \simeq \mathcal{H}_0\oplus \left(
  \bigoplus_{1\le i\le j\le 2}\;\bigoplus_{\mathcal{C} \in
    \Omega_{i,j}/\sim}  \mathcal{H}_{\mathcal{C}} \right)\oplus
\mathcal{H}_{1,\infty} \oplus \mathcal{H}_{2,\infty} \oplus
\mathcal{H}_{\infty,\infty} 
\end{equation}
where the notation is close to the one in \cite[Section
3]{austin09:_count_to_quest_of_atiyah}, namely: 
\[
\mathcal{H}_{\mathcal{C}}=\bigoplus_{(P,R,\psi)\in \mathcal{C}}
\mathcal{H}_{(P,R,\psi)}~~\mathrm{and}~~
\mathcal{H}_{(P,R,\psi)}=\mathrm{Im}(M_{\mathbf{1}_{C(P,R,\psi)}}) 
\]
while $\mathcal{H}_{i,j}=\mathrm{Im}(M_{\mathbf{1}_{C_{i,j}}})$,
$\mathcal{H}_{0}=\mathrm{Im}(M_{\mathbf{1}_{C_{0}}})$.

More precisely, we should have written $C_{i,j}\cap V^\perp$, and we think of
the characteristic function
$\mathbf{1}_{C_{i,j}}$ as a bounded measurable function on $V^\perp$, thus acting
by left multiplication on $L^2(V^\perp)$ and also by twisted left
multiplication on $L^2(V^\perp)\bar\otimes
l^2(\G)$. For notational convenience, we have omitted reference to $V$ here.

Note that, because $\mathcal{H}_{(P,R,\psi)}$ is defined by left
  multiplication of $L^2(V^\perp)\bar\otimes \ell^2(\G)$ with a projection in
  $L^\infty(V^\perp)\rtimes_{\hat \alpha} \G$, it is a right Hilbert
  $\integers_2^{\oplus\G}/V\rtimes\G$-module. Its von Neumann dimension is
  given by the measure of the subset $C_{P,R,\psi}$:
  \begin{equation}\label{eq:dim_of_relevant_subspace}
    \dim_{\integers_2^{\oplus\G}/V\rtimes\G}( \mathcal{H}_{P,R,\psi}) =
    m_{V^\perp}(C_{P,R,\psi}\cap V^\perp).
  \end{equation}

Corresponding statements apply to the other subspaces.

\begin{proposition}\label{prop:decomp_of_op}
  For each $\mathcal{C}\in\Omega_{i,j}/\sim$ (with $1\le i\le j\le 2$) the
  subspace $\mathcal{H}_{\mathcal{C}}$ is $A$-invariant. 

 Moreover, let 
 $V^{l,i,j} $
be the following weighted graphs: a segment of length $l\ge 2$ where each interior
edge has weight two, and if $(i,j)=(1,1)$, both boundary edges have weight
$2$ as well, whereas if $(i,j)=(1,2)$ then one boundary edge has weight $\specialvalue$
while the other has weight $2$, and if $(i,j)=(2,2)$ then both boundary edges
have weight $\specialvalue$.

Let $A^{l,i,j}$
be the  weighted adjacency matrices, regarded as operators on $l^2 (V_v^{l,i,j}
)$, where $V_v^{l,i,j}$ denotes the vertex set of the graph $V^{l,i,j}$. 

Choose one $(P,R,\psi)\in\mathcal{C}$, e.g.~the one with $e$ as the (left)
lower endpoint of $P$; ``left'' if $P$ is a hook (and not a vertical segment).

Then we have a unitary equivalence of
Hilbert $\integers_2^{\oplus\G}/V\rtimes\G$-endomorphisms
$$A |_{\mathcal{H}_\mathcal{C}}  \simeq id_{\mathcal{H}_{(P,R,\psi)}} 
\otimes A^ {l,i,j}$$ 
where $l$ is the length of the path $P$ and $(i,j)$ is the ending scenario of
$(P,R,\psi)$.
\end{proposition}

\begin{proof}
The proof is essentially the same as for the corresponding statement
\cite[Proposition 3.12]{austin09:_count_to_quest_of_atiyah}.

First observe that for $g\in\G$ the operator $T_{g^{-1}}$ (which is a
Hilbert-$\integers_2^{\oplus\G}/V\rtimes\G$ isometry) maps
$\mathcal{H}_{(P,R,\psi)}$ isometrically to
$\mathcal{H}_{(g^{-1}P,g^{-1}R,\hat\alpha_{g^{-1}}\psi)}$. 

This implies that $A$, because of its shape, maps a vector in
$\mathcal{H}_{(P,R,\psi)}$ indeed to a linear combination of vectors in
$\mathcal{H}_{(sP,sR,s\psi)}$ for the generators $s$. However, inspection of the
functions $G_s$ in Proposition \ref{prop:precise_A} or Remark
\ref{rem:colloquial_A} shows that a non-zero contribution is obtained only if
$s\in P$. Consequently, $\mathcal{H}_{\mathcal{C}}$ for
$\mathcal{C}\in\Omega_{i,j}/\sim$ is $A$-invariant. Moreover, inspection of
\ref{prop:precise_A} further shows that $A$ maps one summand to the other
(up to identification with the unitary $T_g$) exactly with the weights as
described by $A^{l,i,j}$; details of the argument follow exactly as in
\cite[Proposition 3.12]{austin09:_count_to_quest_of_atiyah}.

Moreover, Proposition \ref{prop:precise_A} also shows that the operator is
zero on $\mathcal{H}_{(P,R,\psi)}$ if $\abs{R}\le 2$.

Note also that
$A_{|\mathcal{H}_0}=0$.
\end{proof}

\section{The finite dimensional models}

We will concentrate now on the particular eigenvalue $-2$. 

\begin{lemma} \label{4*}
The value $-2$ is an eigenvalue for $A^ {l,1,1}$ 
acting on $l^2 (V^{l,1,1} )$ only for $ l \equiv 1(3)$ and the eigenspace is  one dimensional
in this case.

The value $-2$ is never an eigenvalue for $A^{l,i,2}$ for $l\in\naturals$ and $i\in\{1,2\}$.
\end{lemma}

\begin{proof}
We first study the kernel for the $l\times(l+1)$-matrix $
\left(\begin{smallmatrix}
  2 & \alpha & 0 & \dots\\ \alpha & 2 & 2 & 0 &\dots\\
  0 & 2 & 2 & 2 & 0 & \dots\\
  \dots\\
  & & & \dots & 0 & 2 & 2 & \beta\\
\end{smallmatrix}\right)$ obtained by deleting the last row, and where
$\alpha,\beta\in \{1,\specialvalue\}$. A simple linear recursion shows that
this kernel is $1$-dimensional and spanned by the vector
\begin{equation*}
  (2\alpha, -4,4-\alpha^2, \alpha^2, -4,\ 4-\alpha^2,\dots, x_l),
\end{equation*}
  with 
  \begin{equation*}
x_l=
  \begin{cases}
   2 \alpha^2\beta^{-1} & l\equiv 0 \pmod 3\\
    -8\beta^{-1} & l\equiv 1\pmod 3\\
    2(4-\alpha^2)\beta^{-1} & l\equiv 2\pmod 3.
  \end{cases}  
\end{equation*}
The kernel of $A^{l,i,j}$ is non-trivial if and only if this vector is also
mapped to zero by the last row of $A^{l,i,j}$, which is simply the condition
\begin{equation*}
  \begin{cases}
    \beta(4-\alpha^2)+4\beta^{-1}\alpha^2 =0 & l\equiv 0\pmod 3\\
    \beta\alpha^2-16\beta^{-1}= 0 & l\equiv 1\pmod 3 \\
    -4\beta+4(4-\alpha^2)\beta^{-1}=0 & l\equiv 2\pmod 3.
  \end{cases}
\end{equation*}

If $i=j=1$, i.e.~$\alpha=\beta=2$, this is satisfied if and only if $l\equiv
1\pmod 3$. However, if $\alpha=\specialvalue$ and either $\beta=1$ or
$\beta=\specialvalue$ we check that the condition is never satisfied.
This finishes the proof.
\end{proof}
\section{Dual measures on $\mathbb{Z}_2^{\G}$}\label{sec:measures}

\cite[Lemma 5.1]{austin09:_count_to_quest_of_atiyah} extends readily to our situation:

\begin{lemma}\label{L - 71}
Given a subgroup $V \leq \mathbb{Z}_2^{\oplus \G}$, a finite subset $E \subset
\G$, and $\psi:E\to\mathbb{Z}_2$, let  
$C(\psi)$ be the set of characters $\chi\in\mathbb{Z}_2^{\G}$  such that $\chi_{|E} = \psi$. If $C(\psi)\cap V^\perp \neq \emptyset$, then
\[
m_{V^\perp}(C(\psi)) = {1\over |\{\psi' \in \mathbb{Z}_2^{E}:\ C(\psi')\cap V^\perp \neq \emptyset\}|}.
\]
\end{lemma}

\begin{proof} Given $\psi_1,\psi_2 : E\to\mathbb{Z}_2$, it is enough to show that 
\[
m_{V^\perp}(C(\psi_1))=m_{V^\perp}(C(\psi_2))
\] 
whenever both $C(\psi_1)$ and $C(\psi_2)$ intersect $V^\perp$. Take $\chi_i\in C(\psi_i)\cap V^\perp$. Then translation by
$\chi_2-\chi_1$ sends $C(\psi_1)\cap V^\perp$ to $C(\psi_2)\cap V^\perp$ and preserves the measure $m_{V^\perp}$. 
\end{proof}

We also remark that if $C(\psi)\cap V^\perp = \emptyset$, then certainly $m_{V^\perp}(C(\psi))=0$.

\bigskip

Given a finite subset $F\subset \G$ and a subgroup $\Lambda\subset \G$, we define a left invariant subgroup $V_{F,\Lambda}$ of $\mathbb{Z}_2^{\oplus \G}$ in the following way:
\begin{equation}\label{def_of_V}
V_{F,\Lambda}=\mathrm{span}_{\mathbb{Z}_2} \{ \mathbf{1}_{gF} -
\mathbf{1}_{gtF}\mid \ g\in \G, t\in \Lambda \}, 
\end{equation}
where $\mathbf{1}_F$ is the characteristic function of the set $F$. 

Setting $\chi(F):=\sum_{v\in F} \chi(v)$, we have
 \begin{equation}\label{def_of_Vperp}
V_{F,\Lambda}^\perp =\{ \chi \in \mathbb{Z}_2^{\G}\mid \chi(gF)=\chi(gtF),\;\forall g\in \G, t\in \Lambda \}.
\end{equation} 

We will need furthermore that for certain finite subsets $F,G\subset \G$ as
considered in below the family
\begin{equation}\label{linearly indep}
\{\mathbf{1}_{gF}\mid g\in G\}\text{ is linearly independent over }\mathbb{Z}_2.
\end{equation} 
Since over $\mathbb{Z}_2$ any linear combination of elements in
$\{\mathbf{1}_{gF}\mid g\in G\}$ is of the form
\[
\sum_{g\in H} \mathbf{1}_{gF}=\big(\sum_{g\in H} \mathbf{1}_g\big)\big (\sum_{f\in F} \mathbf{1}_f\big
);\qquad H\subset G,
\]
this linear independence is guaranteed if the group algebra
$\mathbb{Z}_2[\G]$ has no non-trivial zero divisor, 
namely that the zero divisor conjecture for $\G$ holds over
$\mathbb{Z}_2$. This is satisfied for the examples we discuss in this paper
(namely   $\integers\wr\integers$ or the free group) in fact there are no
known torsion--free counterexamples.

\begin{definition}\label{Def:extension}
Let $E,F$ be subsets of $\G$ and $\Lambda\leq \G$ be a subgroup. We say that
$E$  has the \emph{extension property relative to $(F,\Lambda)$} if whenever
$\psi \colon E\to \mathbb{Z}_2$ satisfies
\begin{equation}\label{ext_cond}
\psi(gF)=\psi(gtF)\;\forall g\in \G,t\in \Lambda\;\text{such that }gF\cup gtF\subset E
\end{equation}
then exists $\chi\in V_{F,\Lambda}^\perp$ such that $\chi_{|E}=\psi$.  
\end{definition}

In other words, if the obvious set of conditions on $\psi$ is satisfied on $E$, then $\psi$ extends to an element of $V_{F,\Lambda}^\perp$. Given $E,F,\Lambda,\G$ as in Definition \ref{Def:extension}, we let $\Omega_{F,E}$ be the set 
\[
 \Omega_{F,E}:=\{gF\subset E \mid \ g\in \G\}.
 \] 
We denote by $\Omega_{F,E}/\Lambda$ the set of classes of the equivalence
relation $\sim_\Lambda$ on $\Omega_{F,E}$ given by right multiplication by
$\Lambda$ on $\G$, namely
\[
gF\sim_\Lambda g'F \quad\iff\quad \exists t\in\Lambda\colon
gtF=g'F\quad\iff \quad g^{-1}g'\in \Lambda.
\]
This equivalence holds as $\G$ is torsion free and $F$ is finite.

The following is a generalization of \cite[Corollary
5.9]{austin09:_count_to_quest_of_atiyah} with the same proof.

\begin{lemma}\label{lem:number_of_extendables}
Let $E,F$ be finite subsets of $\G$ and $\Lambda\leq \G$ be a subgroup. Assume that $E$ has the extension property relative to $(F,\Lambda)$. Then
\[
 |\{\psi \in \mathbb{Z}_2^{E}\mid \ C(\psi)\cap V_{F,\Lambda}^\perp \neq \emptyset\}|=2^{|E|-K}
\]
where $K=|\Omega_{F,E}|-|\Omega_{F,E}/\Lambda|$ and with $C(\psi)$ as in Lemma
\ref{L - 71}.
\end{lemma}

\begin{proof}
Since $E$ has the extension property relative to $(F,\Lambda)$, the subset
\[
\{\psi \in \mathbb{Z}_2^{E}\mid \ C(\psi)\cap V_{F,\Lambda}^\perp \neq \emptyset\}
\]
coincides with
\[
\{\psi \in \mathbb{Z}_2^{E}\mid\psi(gF)=\psi(g'F)\, \text{whenever } gF,g'F\in \Omega_{F,E},\, gF\sim_\Lambda g'F\}.
\]
The latter is the annihilator of the finite dimensional subspace
\[
V_{E,F,\Lambda}:=\mathrm{span}_{\mathbb{Z}_2}\{\mathbf{1}_{gF} - \mathbf{1}_{g'F}\mid \  gF,g'F\in \Omega_{F,E},\, gF\sim_\Lambda g'F\}.
\]
By (\ref{linearly indep}) this span is the direct sum over the equivalence
classes $C\in \Omega_{F,E}/\Lambda$ of
$V_C:=\mathrm{span}_{\mathbb{Z}_2}\{\mathbf{1}_{gF} - \mathbf{1}_{g'F}\mid \
gF,g'F\in C\}$ and again by \eqref{linearly indep}
$\dim_{\integers_2}(V_C)=\abs{C}-1$. It follows that
\[
K:=\dim_{\integers_2}(V_{E,F,\Lambda})=\sum_{C\in \Omega_{F,E}/\Lambda} |C|-1=|\Omega_{F,E}|-|\Omega_{F,E}/\Lambda|.
\]
Therefore,
\[
\dim_{\mathbb{Z}_2} \{\psi \in \mathbb{Z}_2^{E}:\ C(\psi)\cap V_{F,\Lambda}^\perp \neq \emptyset\}=|E|-K,
\]
hence the result.
\end{proof}

\begin{corollary}\label{corol:explicit_measure_of_relevant_subspace}
  Assume, in the situation of Section \ref{sec:unitary_equiv}, that
  $V^\perp=V^\perp_{F,\Lambda}$ for $\Lambda\subset\G$ a subgroup and
    $F\subset\G$ finite 
    as above. Moreover, given $\psi\colon U(R,1)\to\integers_2$, assume that
    it extends to $V^\perp$ and that $U(R,1)$ satisfies the extension property
    for $F$. Then
  \begin{equation}
    \label{eq:the_dimension}
        \dim_{\integers_2^{\oplus\G}/V\rtimes\G}( \mathcal{H}_{P,R,\psi}) =
      2^{-\abs{U(R,1)}+K},
  \end{equation}
  where $K$ is the number of equivalence classes of subsets $gF$ of $U(R,1)$
  ($g\in\G$), 
  with $gF\sim gtF$ for $t\in\Lambda$.
\end{corollary}
\begin{proof}
  This is a direct consequence of Lemma \ref{L - 71} and Lemma
  \ref{lem:number_of_extendables}. 
\end{proof}

\section{Extension lemma}

We need a sufficiently general condition for deciding when a set $E$ has the
extension property relative to $(F,\Lambda)$. The following criterion is an
analog of \cite[Lemma 5.5]{austin09:_count_to_quest_of_atiyah}.





\begin{lemma} \label{ext}
Suppose that $s_1$ is of infinite order, $F\subset \{s_1^k\mid
k\in\integers\}\subset\G$ is finite and that
the subset $B \subset \G$ is \emph{horizontally connected}, i.e. 
$ \forall g \in \G$
$$ \{  g  s_1^k  \mid  k \in \mathbb{Z}  \}  \cap B       $$
is a connected segment.

Then $B$ has the extension property (as in Definition \ref{Def:extension}) relative to $(F,\Lambda)$ for any subgroup
$\Lambda$ of $\G$.
\end{lemma}
\begin{proof}
Let $B_n$ be an increasing sequence of subsets of $\G$ such that $B_0 = B$ and such that 
$B_{n+1}$ is obtained from $B_n$ by adding an element at distance 1 from $B_n$,
$B_{n+1}$ is horizontally connected and the union of $B_n$ is $\G$.
We construct a sequence of functions $\chi_n$ such that $\chi_0 = \psi$ and 
$\chi_{n+1}|_{B_{n+1}}  = \chi_n$ and the condition \eqref{ext_cond} is true
with $E=B_n$. This implies the existence
of the extension. Of course it suffices to give a construction of $\chi_1$.

Suppose that we add to $B$ one element $h$ to get $B_1=B\cup\{h\}$ and we want to construct
$\chi_1$. If $gF\subset B_1$
 and $h\in gF$ then by horizontal connectivity of $B$ and the special shape of $F$
the element $h$ is an end-point of $F$.

Then, again by horizontal connectivity of $B$ it is not possible that $h \in
g'F$ for some other $g'$ with $g'F\subset B_1$. 

If there is $ t\in\Lambda \setminus \{  e \}$ such that also $gtF\subset B_1$ then
necessarily $gtF\subset B$ and by~\eqref{ext_cond}
 $\chi_1(h)$ is imposed by
$$
\chi_1(h) = \chi(gF\setminus\{h\})-\chi(gtF).
$$
We only need to show that this is independent of the choice of $t$. Indeed, if for
$t,t'\in \Lambda$ both $gtF,gt'F\subset B$ then, as $t't^{-1}\in\Lambda$ and
because $B$ satisfies condition \eqref{ext_cond},
$\chi(gtF)=\chi(gt'F)$.

If no pair $gF,gtF$ as considered above exist, we can choose $\chi_1(h)$ at
will, e.g.~$\chi_1(h):=0$, as no
additional condition has to be satisfied for \eqref{ext_cond} to hold for $E=B_1$.
\end{proof}





\section{Subgroups $V_I$ and the effect on the eigenspaces}
\label{sec:special_subgroups}

As before, we consider a group $\G$ generated by $s_1,s_2$, but from now on we will mostly 
concentrate on the case of the free group or of $\integers\wr\integers$.

\begin{definition}\label{def:elements}
  For $n\in\naturals = \{ 1, 2, \ldots \}$, define $t_n:=s_2^ns_1s_2^{-n}$. For
  $I\subset\naturals$, define $\Lambda_I:=\langle t_i\mid i\in
  I\rangle\subgroup \G$.

  If $F\subset\G$ is finite, define $V_{F,I}:=\Span_{\integers_2}\{g\mathbf{1}_F-gt \mathbf{1}_F\mid g\in \G, t\in \Lambda_I \}\subset
  \integers_2^{\oplus\G}$ so its
  Pontryagin dual is
  \begin{equation*}
V_{F,I}^\perp =\{ \chi \in \mathbb{Z}_2^{\G}\mid
  \chi(gF)=\chi(gtF),\;\forall g\in \G, t\in \Lambda\}
\end{equation*}
 as in Definition
  \eqref{def_of_Vperp}. Finally, we specialize to $F_l=\{s_1^{-1},e,s_1\}$ and
  set $G_{I}:= \integers_2^{\oplus \G}/V_{F_l,I}\semiprod\G$. 
\end{definition}

\begin{lemma}\label{lem:relations}
  For 
  \begin{equation}\label{eq:ZwreathZ}
 \G=\integers\wr\integers=\langle s_1,s_2\mid
  [s_2^ks_1s_2^{-k},s_1]=1\; \forall k\in\integers\rangle
  \end{equation}
 the elements $t_n$
  are the free abelian generators of a free abelian subgroup, equal to the
  kernel of the obvious projection to $\integers = \langle s_2\rangle$. 

  For $\Gamma$ the free group on free generator $s_1,s_2$, the elements $t_n$
  are the free generators of a free subgroup.

  In particular, if $0\notin I$ then $\Lambda_I$ intersects $\langle
  s_1\rangle$, the subgroup generated by $s_1$, only in the trivial element.
\end{lemma}
\begin{proof}
  For $\integers\wr\integers$ this is part of the structure theory of the
  wreath product: the base 
  $\integers^{\oplus\integers}$ is a free abelian group with generators
  $s_2^ns_1s^{-n}_2$ for $n\in\integers$. The group $\langle
  s_2\rangle=\integers$ acts on the base by the obvious permutation of the
  basis elements, the semi-direct product is $\integers\wr\integers$.

  For the free group, the assertion follows from an easy normal form
  calculation, which is carried out in detail in \cite[Lemma
  5.2]{austin09:_count_to_quest_of_atiyah}. 
\end{proof}

\begin{proposition}\label{prop:identifications_in_hooks}
  Assume that $P=\left \{
gs_2^{-n}, \ldots, g , g s_1, g s_1 s_2^{-1}, \ldots , g s_1 s_2^{-m}
\right \}$ is a hook as in \eqref{eq:def_of_hook} with $n,m\in\naturals$
and $I\subset\naturals$ is given. Assume moreover that $\G$ is either the free
group on free generators $s_1,s_2$ or $\G=\integers\wr\integers$ as in
\eqref{eq:ZwreathZ}.

Then $xt=y$ for $x\ne y\in P$ and $t\in\Lambda_I$ exactly if $t=t_k$ for some
generator $t_k$ with $k\in I$ such that $k\le n$ and $k\le m$,
$x=gs_2^{-k}$, $y=gs_1s_2^{-k}$ (or $t=t_k^{-1}$,  $x=gs_1s_2^{-k}$,
$y=gs_2^{-k}$). 
\end{proposition}
\begin{proof}
  Obviously, the $t_k,x,y$ we have given satisfy all the conditions in both
  cases.

  We next show that the conditions are necessary. Write $x=g s_1^{\epsilon_x} s_2^{a_x}$ and
  $y=gs_1^{\epsilon_y}s_2^{a_y}$ in $P$, with $\epsilon_x,\epsilon_y\in
  \{0,1\}$ and $a_x,a_y\in\integers$. If $x^{-1}y\in\Lambda_I$,
  then 
  in particular $x^{-1}y=s_2^{-a_x} s_1^{\epsilon_y-\epsilon_x}s_2^{a_y}$ is
  mapped to the trivial element under the 
  projection to the infinite cyclic group $\langle s_2\rangle$ which maps
  $s_1$ to $1$. It 
  follows that $a_y=a_y$, so $x^{-1}y=s_2^{-a_x} s_1^{\epsilon_y-\epsilon_x}s_2^{a_x}$.
 In the two cases we study, $\Lambda_I$ is contained either in the free or the
 free
 abelian groups on generators $s_2^vs_1s_2^{-v}$ 
  ($v\in\integers$). The assertion now follows.
\end{proof}

\begin{proposition}\label{lem:extensions_exist}
  As before, assume $\G$ generated by $s_1,s_2$ is either
  $\integers\wr\integers$ or the free group on $s_1,s_2$. Set $F_l=\{s_1^{-1},e,s_1\}$ and
  let $R\subset\G$ be a hook, $\psi\colon U(R,1)\to\integers_2$ the
  characteristic function of $R$. Then $\psi$ extends to
  $V^\perp_{F_l,I}$. Moreover, $gF_l,hF_l\subset U(R,1)$ are equivalent if and only
  if $g,h\in R$ and there is $t\in \Lambda_I$ with $g=ht$.
\end{proposition}
\begin{proof}
  By a normal form argument, we know that whenever $gF_l\subset U(R,1)$ then
  $g\in R$. By Lemma \ref{ext} we only have to check that, whenever $xF_l$ and
  $yF_l$ for $x,y\in R$ are equivalent, then $\psi(xF_l)=\psi(yF_l)$. By Proposition
  \ref{prop:identifications_in_hooks} (and normal form in $\G$), if $xF_l$ and $yF_l$
  are equivalent then $\abs{xF_l\cap U(R,1)}=1=\abs{yF_l\cap U(R,1)}$: the
    intersection would have different cardinality only if $x$ (or $y$) was part of the
    ``bend'' of the hook, i.e.~$xs_1$ or $xs_1^{-1}\in R$, but then $x$ is
    only equivalent to itself. Finally $\psi(xF_l)\equiv\abs{xF_l\cap U(R,1)} \pmod 2$,
      therefore $\psi(xF_l)=\psi(yF_l)$ and the proposition follows.
\end{proof}

\begin{corollary}\label{coro:final_dim}
  Adopt the situation of Proposition \ref{lem:extensions_exist}. Assume that
  $n$ and $m$ are the length of the left and of the right leg of $R$,
  respectively. Then, with $K:=\abs{I\cap \{1,\dots,\min(m,n)\}}$,
  \begin{equation}\label{eq:final_dim}
      \dim_{(\integers_2^{\oplus\G}/V_{F_l,I})\rtimes\G}( \mathcal{H}_{R,R,\psi}) =
      2^{-3(n+m)-8+K}. 
  \end{equation}
\end{corollary}
\begin{proof}
  Because of Proposition \ref{lem:extensions_exist} and Lemma \ref{ext}, we can directly apply
  Corollary \ref{corol:explicit_measure_of_relevant_subspace}.

  By normal form, we know that inside $U(R,1)$ there are no relations and that
  $\abs{R}=n+m+2$, hence $\abs{U(R,1)}=3(n+m+2)+2$. Moreover, by Propositions
  \ref{lem:extensions_exist} and \ref{prop:identifications_in_hooks}, the
  correction term $K$ is exactly as given. 
\end{proof}


We conclude this section by showing that in the situation of Proposition  \ref{lem:extensions_exist}  the sets $C_{1,\infty}$,
$C_{2,\infty}$, and $C_{\infty,\infty}$, as defined in Section \ref{sec:decomposition_of_V}, are
negligible with respect to $m_{V_{F_l,I}^\perp}$.  

\begin{lemma}\label{Lemma:vperpinfinite}  
Given $g\in \G$, let $D_g$ be the measurable subset of $\mathbb{Z}_2^{\G}$ defined by 
\[
D_g=\{\chi\in \mathbb{Z}_2^{\G} \mid \chi(gs_2^{-k})=1~\mathrm{and}~\chi(gs_2^{-k}s_1^{a})=0,~~\forall k\geq 0,~ a=\pm 1\}.
\] 
Then $m_{V_{F_l,I}^\perp}(D_g)=0$.
\end{lemma}

\begin{proof}
Given $g\in \G$ and an integer $N\geq 0$ set 
\[
D_{g,N}=\{\chi\in \mathbb{Z}_2^{\G} \mid
\chi(gs_2^{-k})=1~\mathrm{and}~\chi(gs_2^{-k}s_1^{a})=0,~~\forall 0\leq k\le  N,~ a=\pm 1\}
\]
Let
$F=F_l=\{s_1^{-1},e,s_1\}$ and $E=\{s_2^{-k}s_1^{-2}, s_2^{-k},
s_2^{-k}s_1\mid k=0,\dots,N-1\}$.  The number
  of ways to embed shifted copies of $F$ into $E$ is $\abs{\Omega_{E,F}}=N$.
  By Lemma \ref{L - 71} and \ref{lem:number_of_extendables} we have 
\[
\mu_{V_{F_l,I}^\perp}(D_{g,N})\leq {1\over 2^{3N-\abs{\Omega_{E,F}|}}} =
\frac{1}{2^{2N}} .
\]
Since 
\[
D_g=\bigcap_{N\geq 1} D_{g,N}, 
\]
we obtain $\mu_{V_{F_l,I}^\perp}(D_g)\leq 2^{-2N}$ for all $N\geq 1$ and the lemma follows.
\end{proof}

Thus, we get the following analog of \cite[Prop. 5.8]{austin09:_count_to_quest_of_atiyah}.

\begin{corollary}\label{corol:vanish_of_infinite}
Keeping the notations above, we have
\[
m_{V_{F_l,I}^\perp}(C_{1,\infty})=m_{V_{F_l,I}^\perp}(C_{2,\infty})=m_{V_{F_l,I}^\perp}(C_{\infty,\infty})=0.
\] 
In particular,
\begin{equation*}
  \mathcal{H}_{1,\infty}=\{0\}=\mathcal{H}_{2,\infty}=\mathcal{H}_{\infty,\infty}.
\end{equation*}

\end{corollary}

\begin{proof}
Indeed, $C_{1,\infty}\cup C_{2,\infty}\cup C_{\infty,\infty}\subset \bigcup_{g\in \G} D_g$,
so we may apply Lemma \ref{Lemma:vperpinfinite}. The second claim follows
directly from the corresponding version of \eqref{eq:dim_of_relevant_subspace}.
\end{proof}

\begin{remark}
Lemma \ref{Lemma:vperpinfinite} and its corollary above can be extended to
almost arbitrary subgroups $V_{F,\Lambda}$, where  $F\subset \G$ is a finite
subset, $\Lambda\leq \G$ is a subgroup as in Section \ref{sec:measures}, and 
$E\subset \G$ is any set having the extension
property with respect to $(F,\Lambda)$. For instance it is enough to have that
$|F|\geq 2$ and that the subsets $gF$ ($g\in \G$) which are included in $E$
are pairwise disjoint (as the proof of Lemma \ref{Lemma:vperpinfinite} 
shows).
\end{remark}

\section{Explicit calculation of the von Neumann dimension of the eigenspace}

We continue with the situation of Section \ref{sec:special_subgroups}.
We deal only with the eigenvalue $-2$ for $A$ and we set $Q= A+2$.
\begin{theorem}\label{theo:main}
  Fix $I:=\{2,n_1,n_2,\dots\}\subset\naturals$ with $n_0:=2<n_1<n_2<\dots$ and such
  that $n_k\equiv 2\pmod 3\;\forall k$. Choose
  $\G=\integers\wr\integers$ or $\Gamma$ free on two generators and set
  $G_I=(\integers_2^{\oplus\G}/V_{F_l,I})\semiprod\G$ as in Definition
  \ref{def:elements}. Construct $A\in \rationals[G_I]$ as above. Then
$$ \dim_{G_I}  \ker (A+2) 
=  \beta_1 + \beta_2   \sum_{k=1}^\infty   2^{-6 n_k +k}. $$
Here,
$\beta_1$ and $\beta_2$ are explicitly given rational numbers, compare
\eqref{eq:formula}.

In particular, these numbers show up as $L^2$-Betti numbers of normal
coverings of compact manifolds with covering group $G_I$.
\end{theorem}
\begin{proof}
We use Proposition \ref{prop:decomp_of_op} to decompose $A$. By Corollary
\ref{corol:vanish_of_infinite} and Lemma \ref{4*} the only contributions to the
eigenvalue $-2$ are obtained on $\mathcal{H}_{\mathcal{C}}$ if $\mathcal{C}\in
\Omega_{1,1}/\sim$ and if the length of the associated hook $R$ is congruent $1$
modulo $3$, and the $L^2$-dimension of the eigenspace is then
\begin{equation}\label{eq:repeat_contrib}
\dim_{G_I}(\mathcal{H}_{R,R,\psi})=2^{-3(l_1+l_2)-8+\abs{I\cap
    \{1\dots,\min\{l_1,l_2\}\}}},
\end{equation}
where $R$ is the hook, $\psi$ is the characteristic function of the hook in its
$1$-neighborhood and $l_1,l_2$ are the lengths of the left and right leg of the hook,
respectively. Note that the length of the hook is $l_1+l_2+1$, so we get a
contribution exactly if $l_1+l_2$ is divisible by $3$.

Write $I=\{2,n_1,n_2,\dots\}$ with $2<n_1<n_2<\dots$. We have to add the
summand \eqref{eq:repeat_contrib} for each $1\le l_1,l_2$ with $l_1+l_2$
divisible by $3$ (each such corresponding to one class of hook passing through
$e$). To facilitate the effect of $\abs{I\cap
    \{1,\dots,\min\{l_1,l_2\}\}}$, 
we choose the disjoint decomposition of the
  $(l_1,l_2)$-plane into subsets $V_k:=U_k\setminus U_{k+1}$, where
  $U_k=\{(l_1,l_2)\mid l_1,l_2\ge n_k\}$, such that $\abs{I\cap
    \{1,\dots,\min\{l_1,l_2\}\}}=k$ on $V_k$.

We obtain (with convention $n_0=2$)
\begin{equation}\label{eq:U_k_sum}
  \begin{split}
    &\dim_{G_I} \ker (A+2)  = 2^{-8}\sum_{k=0}^\infty 2^k
    \sum_{\small \begin{array}{c} (l_1,l_2)\in V_k\\
        l_1 + l_2 \equiv 0 (3) \end{array}} 2^{-3(l_1+l_2)}\\
&  =
    2^{-8}\sum_{k=0}^\infty 2^k \left(  \sum_{ \begin{array}{c} (l_1,l_2)\in U_k\\
          l_1 + l_2   \equiv 0 (3)  \end{array}} 2^{-3(l_1+l_2)} -   \sum_{\small \begin{array}{c} (l_1,l_2)\in U_{k+1}\\
          l_1 + l_2 \equiv 0 (3) \end{array}} 2^{-3(l_1+l_2)}\right).
  \end{split}
\end{equation}

Recall that all $n_k$ are congruent $2$ modulo $3$. We distinguish the cases
$l_1=3r_1+r$ with $r=0,1,2$ (and $l_2=3r_2+2-r$ to get $l_1+l_2\equiv
0\pmod 3$) and obtain finally for the sum over $U_k$
\begin{equation*}
  \begin{split}
    \sum_{\small \begin{array}{c} (l_1,l_2)\in U_k\\
        l_1 + l_2 \equiv 0 (3) \end{array}} 2^{-3(l_1+l_2)} & = \sum_{r=0}^2
    \sum_{r_1=0}^\infty 2^{-3(n_k+3r_1+r)}\sum_{r_2=0}^\infty
    2^{-3(n_k+3r_2+2-r)}\\
& = 3\cdot 2^{-6n_k+2}(1-2^{-9})^{-2}.
  \end{split}
\end{equation*}

Substituting this in \eqref{eq:U_k_sum} we get
\begin{equation}\label{eq:formula}
  \begin{split}
    \dim_{G_I} \ker (A+2) &= \frac{3}{2^{6}(1-2^{-9})^2}\sum_{k=0}^\infty
    \left(2^k 2^{-6n_k}-\frac{1}{2}2^{k+1} 2^{-6n_{k+1}}\right) \\
& =
    \frac{3}{2^6(1-2^{-9})^2}2^{-12} + \frac{3}{2^7(1-2^{-9})^2}
    \sum_{k=1}^\infty 2^{-6n_k+k}.
  \end{split}
\end{equation}

 \end{proof}

\section{Arbitrary real numbers as  $L^2$-Betti numbers for normal coverings}

Our main point about the explicit formulas for $L^2$-Betti numbers is
two-fold: on the one hand, we want to show that every positive real number is
an $L^2$-Betti number. This is the goal of the current section.

Secondly, we want to show that we get transcendental $L^2$-Betti numbers
for \emph{universal} coverings, which translates algebraically that we have to
use finitely presented groups. This will be done in the last sections.

Now we show how, starting from the $L^2$-Betti numbers we
explicitly obtain in Theorem \ref{theo:main}, one can construct (again
explicitly) more groups and 
elements in their group rings to finally get the following theorems.

\begin{theorem}
  For every real number $r\ge 0$ their is a finitely generated group
  $\Gamma_r$, an $l\in\naturals$ and $a_r\in M_l(\integers\Gamma_r)$ such that
  \begin{equation*}
    \dim_{\Gamma_r}(\ker(a_r)) =r 
  \end{equation*}
and from a dyadic expansion $r=\sum \lambda_j 2^j$ with $\lambda_j\in
  \{0,1\}$ we obtain (in principle) an ``explicit'' description of $\Gamma_r$
  and $a_r$.

   Moreover, there is a compact manifold $M$ with a normal covering $\tilde
   M$ (with covering group $\Gamma_r$) such that
   \begin{equation*}
     b_3^{(2)}(\tilde M;\Gamma_r) =r.
   \end{equation*}
\end{theorem}

To prove this from the previous constructions, we review a
couple of constructions for which we can control the $L^2$-Betti numbers in
terms of $L^2$-Betti numbers of the ingredients.

\begin{lemma} \label{sum}
  Let $\Gamma_1,\Gamma_2$ be two groups, $l_1,l_2\in\naturals$ and $a_j\in
  M_{l_j}(\integers[\Gamma_j])$ for $j=1,2$. Form the ``block sum''
  \begin{equation*}
    a:=a_1\oplus a_2\in M_{l_1+l_2}(\integers[\Gamma_1\times \Gamma_2]),
  \end{equation*}
  where we tacitly identify $\Gamma_j$ with its image in
  $\Gamma:=\Gamma_1\times \Gamma_2$ and identify $a_j$ with its image under
  the induced map. Then
  \begin{equation*}
    \dim_{\Gamma}(\ker(a)) = \dim_{\Gamma_1}(\ker(a_1)) +
      \dim_{\Gamma_2}(\ker(a_2)). 
  \end{equation*}
\end{lemma}
\begin{proof}
  This is well known and essentially clear. First of all, by the induction
  principle (e.g.~\cite[Proposition 3.1]{MR1828605}),
  $\dim_{\Gamma}(\ker(a_j))=\dim_{\Gamma_j}(\ker(a_j))$ for $j=1,2$, where we
  think of $a_j$ either as living over $\integers[\Gamma]$ or over
  $\integers[\Gamma_j]$. 

  Secondly, the kernel of $a$ (as block sum) is the direct sum of the kernels
  of $a_1$ and of $a_2$ (in $l^2(\Gamma)^{l_1+l_2}$). As the von Neumann
  dimension is additive for direct sums, the assertion follows.
\end{proof}

\begin{lemma} \label{product}
  Let $\Gamma_1,\Gamma_2$ be two groups, $l_1,l_2\in\naturals$ and $a_j\in
  M_{l_j}(\integers[\Gamma_j])$ for $j=1,2$. Assume that $a_1$ and $a_2$ are
  non-negative (if necessary, replace them by $a_j^*a_j$). Form the ``tensor
  sum''
  \begin{equation*}
    a:=a_1\tensor \id + \id\tensor a_2 \in M_{l_1\cdot
      l_2}(\integers[\Gamma_1]\tensor \integers[\Gamma_2]),
  \end{equation*}
  thinking of $\integers[\Gamma]=\integers[\Gamma_1]\tensor
  \integers[\Gamma_2]$ acting on
  $l^2(\Gamma_1\times\Gamma_2)=l^2(\Gamma_1)\overline{\tensor} l^2(\Gamma_2)$.  Then
  \begin{equation*}
    \dim_{\Gamma}(\ker(a)) = \dim_{\Gamma_1}(\ker(a_1))\cdot
    \dim_{\Gamma_2}(\ker(a_2)). 
  \end{equation*}
\end{lemma}
\begin{proof}
  This lemma is also well known and follows from the fact that in this
  situation $\ker(a)=\ker(a_1)\tensor \ker(a_2)$. A detailed argument for a
  special case can be found in the proof of \cite[Theorem 4.1]{MR1934693}.

  For the sake of completeness, let us give a more explicit proof here. If
  $p_j$ is the orthogonal projection onto $\ker(a_j)$ for $j=1,2$ (considered
  as matrices over $\mathcal{N}\Gamma$, induced up from
  $\mathcal{N}\Gamma_j$), we claim 
  that in this situation $p:=p_1\tensor p_2$ is the projection onto the kernel of
  $a$. As $(a_1\tensor \id+\id\tensor a_2)(p_1\tensor p_2)=0$, the image of
  $p$ is contained in the kernel of $a$.

  Now, $(1-p_1)\tensor p_2 + p_1\tensor (1-p_2) + (1-p_1)\tensor (1-p_2)$ is an orthogonal
  decomposition of $1-p_1$. On the image of $(1-p_1)\tensor p_2$, which is equal
  to $\im(1-p_1)\tensor \im(p_2)$, $a$ coincides with $a_1\tensor \id$ which
  is $>0$ there, and the corresponding argument applies to the image of
  $p_1\tensor (1-p_2)$.

  On the image of $(1-p_1)\tensor (1-p_2)$ which coincides with
  $\im(1-p_1)\tensor \im(1-p_2)$, $a$ coincides with $a_1\tensor
  \id+\id\tensor a_2$, and both summands are $>0$. Altogether, on the
  complement of $\im(p)$ $a>0$ and therefore $\ker(a)=\im(p)$.

  Finally, we have to compute the $\Gamma$-trace of $p$.  Let
  $e_1,\dots,e_{l_1}$ be the standard basis vectors of $l^2(\Gamma_1)^{l_1}$
  and $f_1,\dots, f_{l_2}$ be the standard basis vectors of
  $l^2(\Gamma_2)^{l_2}$ (the characteristic function of the neutral element in
  the corresponding component).

  Then $\{e_i\tensor f_j\}_{i=1,\dots,l_1;\, j=1,\dots , l_2}$ will be the
  standard basis for $l^2(\Gamma_1\times \Gamma_2)^{l_1\cdot l_2}$. Consequently
  \begin{multline*}
    \tr_\Gamma(p)=\sum_{i=1}^{l_1}\sum_{j=1}^{l_2}\innerprod{p_1\tensor
      p_2(e_i\tensor e_j),e_i\tensor e_j}_{l^2(\Gamma_1)\tensor
      l^2(\Gamma_2)}\\
    = 
    \sum_{i=1}^{l_1}\innerprod{p_1(e_i),e_i}_{l^2(\Gamma_1)}\cdot
      \sum_{j=1}^{l_2} \innerprod{p_2(f_j),f_j}_{l^2(\Gamma_2)} =
      \tr_{\Gamma_1}(p_1)\cdot \tr_{\Gamma_2}(p_2) 
  \end{multline*}
  This proves the claim.
\end{proof}

\begin{proposition}\label{prop:set_closure}
  Let $U\subset\reals_{\ge 0}$ be a subset of the non-negative real numbers
  with the following properties
  \begin{enumerate}
  \item $U$ is closed under multiplication with and addition of non-negative
    rational numbers;
  \item $U$ is additively closed: if $r,s\in U$ then also $r+s\in U$; 
  \item there are rational numbers $a,q\in\rationals_{\ge 0}$, $q>0$ and
    $d\in\naturals$ such that for every increasing sequence $0\le
    n_1<n_2<\dots$ the number $a+q\sum_{k=1}^\infty 2^k 2^{-dn_k} \in U$.
  \end{enumerate}
  Then $U=\reals_{\ge 0}$.
\end{proposition}
\begin{proof}
  Choose $m\in\naturals$ such that $b:=2^{dm-1}q>a$ is a multiple of $q$. Adding
  the rational number $b-a$ and multiplying with the rational number
  $2q^{-1}$ we see that all real numbers of the form
  \begin{equation}
    \label{eq:form1}
    2^0\cdot 2^{dm} +\sum_{k=1}^\infty 2^k 2^{-dn_k} \in U;\qquad 0\le
    n_1<\dots .
  \end{equation}
  Replacing $d$ by $D:=dm$ and using only sequences where each $n_k$ is a
  multiple of $m$, and multiplying with suitable powers of $2$, we see that
  all real numbers of the form 
  \begin{equation}
    \label{eq:form_exist}
    \sum_{k=0}^\infty 2^k 2^{-Dn_k}; \qquad 0\le n_0<n_1<\dots
  \end{equation}
  belong to $U$.

  Because $U$ is closed under multiplication with non-negative rational
  numbers it suffices to show that $U$ contains some non-empty open interval. 

  Moreover, because $U$ is additively closed and closed under multiplication
  with powers of $2$, it suffices to show that $U$ contains every real number of
  the form
  \begin{equation}\label{eq:form}
    r= \sum_{n\in I} 2^{-Dn};\qquad I\subset \naturals
  \end{equation}
  since an arbitrary real number between $0$ and $1$ is a sum of at most $D$
  multiples (by $2^k$ with $0\le k<d$) of numbers of the form
  \eqref{eq:form}.

  Fix therefore $I\subset \naturals$. We now describe $2^{D-1}$ numbers of the
  form \eqref{eq:form_exist} with sum equal to $r$.

  Instead of writing down the formulas, we describe the digits of these
  numbers in dyadic expansion. Note that the relevant feature of any number of
  the form \eqref{eq:form_exist} is that the consecutive digits occur at
  places which are multiples of $D$ (as is true for $r$), but each new digit
  shifted one further ``to the left''.

  The first $2^{D-1}$ dyadic digits of $r$ each give one (the first) digit of
  the $2^{D-1}$ numbers to be constructed. The next digit of $r$ (the summand
  $2^{-Dn_v}$ with $v=2^{D-1}+1$), shifted by
  $2^{D-1}$ to the right, gives the second digit of each of the $r$ numbers to
  be   constructed. Note that this is a summand of the form $2^1\cdot
  2^{-D(n_v+1)}$. Note also that the sum of these $2^{D-1}$ summands is
  exactly $2^{-Dn_v}$, i.e.~the corresponding digit of $r$. The next two
  digits are used, shifted by $2^{D-2}$ in the first or last $2^{D-2}$,
  respectively, of our
  numbers to be constructed. The same reasoning as before shows that these
  summands have the right form and add up to the right digits of $r$. 
  The next $4$ digits, shifted by $2^{D-3}$, are used in one quarter each,
  i.e.~$2^{D-3}$, of our numbers to be constructed. 

  We continue this construction inductively until we arrive at $2^{D-1}$
  digits which are not to be shifted at all. Then we cyclically continue this
  pattern inductively.

  The result are by construction the $2^{D-1}$ numbers, each of the form
  \eqref{eq:form_exist}, which therefore belong to $U$ and which add up to
  $r$. 

As explained above, this implies the assertion.
\end{proof}

\begin{corollary}
  Every non-negative real number is an $L^2$-Betti number of some covering of
  a compact manifold.
\end{corollary}
\begin{proof}
  By a standard reduction, it suffices that for every $r\in\reals_{\ge 0}$
  there is a finitely generated group $\Gamma$, $d\in\naturals$ and $A\in
  M_n(\integers\Gamma)$ such that $\dim_{\Gamma}(\ker(A))=r$.

  However, the main result of this paper asserts that 
for $I=\{2,n_1,n_2,\dots\}$ with
$n_0=2< n_1<n_2<\dots$ all congruent $2$ modulo $3$ and for certain $\beta_1,\beta_2\in \rationals_{>0}$,
whenever 
\begin{equation*}
  r=\beta_1+\beta_2\sum_{k=1}^\infty 2^k \cdot 2^{-n_kd}
\end{equation*}
there is a finitely generated $\Gamma_r$ and $a_r\in
\integers[\Gamma_r]$ such that $\dim_{\Gamma_r}(\ker(a_r))=r$.

  Using in addition Lemmas \ref{sum} and \ref{product} the set of von Neumann
  dimensions of kernels satisfies the assumptions of
  Proposition \ref{prop:set_closure}. The corollary follows.
\end{proof}

\section{Structure of the groups $G_I$}

To show that there are also \emph{universal} coverings with transcendental
$L^2$-Betti numbers---equivalently matrices over the group ring of a finitely
presented group with transcendental $L^2$-dimension of the kernel, we have to
analyze the groups used in Theorem \ref{theo:main} more precisely.

Recall from Definition \ref{def:elements} that, starting with $\G=\langle s_1,s_2\rangle$ either free or
$\G=\integers\wr\integers$ and given a subset $I\subset\naturals$, fixing
$F_l=\{s_1^{-1},e,s_1\}$, we have groups
\begin{equation*}
  G_I:= \left(        \mathbb{Z}_2^{\oplus {\G}}    / V_{F_l,I}\right)\semiprod \G.
\end{equation*}

Consider the basis $\{\delta_g\mid g\in \G\}$ of
$\integers_2^{\oplus\G}=\integers_2[\G]$ where $\delta_g$ is
the characteristic function of $\{g\}\subset \G$.

Set $u:=\sum_{h\in F_l}\delta_h$.
Recall that $V_{F_l,I}$ is the $\integers_2[\G]$-submodule of
$\integers_2[\G]$ generated by the
elements of the form $\sum_{h\in F_l}(\delta_{h}- \delta_{th})=u-tu,\; t\in \Lambda_I$,
so as $\integers_2$-vector space it is generated by elements of the form
\begin{equation*}
  \sum_{h\in F_l}(\delta_{gh}- \delta_{gth})=gu-gtu,\qquad g\in \G,\quad t\in
  \Lambda_I. 
\end{equation*}

\begin{lemma}\label{lem:normalform}
  \begin{enumerate}
  \item\label{item:generation} The subgroup $V_{F_l,I}$ is generated as
    $\integers_2[\G]$-module by
    the elements $w_g:=gu-u = \sum_{h\in F_l}( \delta_{h}-\delta_{g h})$ with
    $g=t_n$, $n\in I$.
  \item\label{item:linindep} The translates of $u$ by powers of $s_1$, i.e.~$\{gu\mid g\in \langle
    s_1\rangle\}$ satisfy the following property: if the support of
    $\sum_{k=1}^n s_1^{a_k} u$ with $a_1<a_2<\cdots < a_n$ belongs to
    $\{s_1^l,s_1^{l+1},\dots, s_1^r\}$ then $l+1\le a_1$ and $a_n\le r-1$.

In particular, the $g u$ with $g\in \langle s_1\rangle$ 
form a linearly independent subset of the vector space
    $\integers_2[\langle s_1\rangle]\subset \integers_2[\G]$.
  \item \label{item:cancel}
    Write $y\in V_{F_l,I}$ as sum $y=\sum_i g_i (t_iu-u)$ with $g_i\in \G$,
    $t_i\in\Lambda_I$ and with minimal number of such summands.

  Fix a left coset $g\langle s_1\rangle$ of $\langle s_1\rangle$ and assume
  that the support of $y$ (considered as a function on $\G$) is contained in
  $g\{s_1^l, s_1^{l+1},s_1^{l+2},\ldots, s_1^r\}$ with $l\le r$.

  Then, in the (minimal) sum $y=\sum_i g_i (t_iu-u)$, if $g_i$ or $g_it_i\in
  g\langle s_1\rangle$, then they already lie in
  $g\{s_1^{l+1},\ldots,s_1^{r-1}\}$. 

  In particular, 
 if the support of $y$
does intersect
  the coset $g\langle s_1\rangle$, then $r-l\ge 2$. On the other hand, if the
  support of $y$ does not intersect the coset $g\langle s_1\rangle$, then
  none of the $g_it_i$ and $g_i$ belongs to the coset $g\langle s_1\rangle$. 
 
\item\label{item:where_is_g}
    If $g\in \G$ with $w_g=gu-u\in V_{F_l,I}$ then $g\in\Lambda_I$.
  \end{enumerate}
\end{lemma}
\begin{proof} By definition, $V_{F_l,I}$ is generated as
  $\integers_2[\G]$-module by the $w_g$ with $g\in
  \Lambda_I$. However, 
  \begin{equation*}
    w_g+gw_{g'}= \sum_{h\in F_l} (\delta_{h}-\delta_{gh} +
    \delta_{gh}-\delta_{gg'h}) = w_{gg'}.
  \end{equation*}
  As $\Lambda_I$ by definition is generated by $\{t_n\mid n\in I\}$,
  \ref{item:generation} follows.

  As in our case $s_1$ is infinite cyclic, \ref{item:linindep} is a well known
  statement about $\integers_2[\integers]$: if $x:=\sum_{k=1}^n s_1^{a_k}u$ with
  $a_1<\dots <a_n$,
  then the value of $x$ at $s_1^{a_1-1}$ is non-zero, so $l\le a_1-1$, and
  similarly $r\ge a_n+1$.

  To prove \ref{item:cancel}, fix an arbitrary element $y=\sum_i g_it_i u
  -g_iu$ as in \ref{item:cancel}. 
  Consider now all the summands such that $g_it_i$ or $g_i\in g\langle
  s_1\rangle$. These are exactly the summands in $y$ contributing with one or
  two summands of the form $g s_1^{a_i} u$ whose support is contained in
  $g\{s_1^l,\dots,s_1^r\}\subset g\langle s_1\rangle$. By \ref{item:linindep},
  all the summands with $a_i\le l$ have to 
  appear pairwise to cancel each other out.
  But if we would e.g.~have $g_i=g_j$ for
  $i\ne j$ then we 
  could write $g_it_iu-g_i u + g_jt_ju-g_ju= (g_it_j)(t_j^{-1}t_i) u - (g_it_j)
  u$ with $t_j^{-1}t_i\in\Lambda_I$ and $g_it_j\in \G$, thus being able to
  write $y$ with fewer summands, violating the minimality for the expression
  of $y$. The same reasoning rules out that $g_it_i=g_j$ with $i\ne j$, where we
  obtain $g_it_iu-g_iu + g_jt_ju-g_ju=g_i(t_i^{-1}t_j)u-g_iu$. Finally,
  terms with $g_it_i=g_i$ by minimality also don't appear. 

  Similarly, we can rule out that for a minimal expression $g_i=gs_1^{a_i}$ or
  $g_it_i=gs_1^{a_i}$ with $a_i>r-1$.
  
  The above argument also shows that if $g_i t_i$ or $g_i \in \langle s_1  \rangle$ then the
  summands of the form $g s_1^{a_i} u$ in $y$ do not cancel and therefore 
  $y$ intersects the coset $g \langle s_1 \rangle$.

 To prove \ref{item:where_is_g} observe that the support of $w_g$ does
 intersect only the cosets 
 $\langle s_1\rangle$ and $g\langle s_1\rangle$. Written with minimal number
 of summands as in \ref{item:cancel} therefore
 \begin{equation}\label{eq:gu-u}
gu-u=\sum t_i s_1^{a_i} u-
 s_1^{a_i} u\quad\text{with}\quad t_i\in\Lambda_I
\end{equation}
such that we have equal cosets $t_i\langle
 s_1\rangle= g\langle s_1\rangle$. If $g\in\langle s_1\rangle$, then
 $\Lambda_I$ contains a non-trivial power of $s_1$ and therefore by Lemma
 \ref{lem:relations} contains $s_1$, so $g\in\Lambda_I$. Otherwise,
 $\Lambda_I\cap \langle s_1\rangle=\{e\}$ and by \ref{item:linindep} and
 minimality, the above expression \eqref{eq:gu-u} for $gu-u$ consists of
 exactly one summand $gu-u=ts_1^au-s_1^au$ which finally implies $a=0$ and $g=t\in\Lambda_I$.
\end{proof}

\begin{theorem}\label{theo:recursive_presentation}
 The group $G_I$ has a
  recursive presentation if and only if $I$ is recursively enumerable,
  i.e.~if there is a Turing machine listing exactly all elements of $I$. 
\end{theorem}
\begin{proof}
  Assume that $I$ is recursively enumerable.

  Using Lemma \ref{lem:normalform}, a presentation of $G_I$ is given by the
  generating set $s_1,s_2,\tau=:\delta_e$ with the following relations:
  \begin{itemize}
  \item $\tau^2=1$
  \item $g^{-1}\tau g=:\delta_g$ commutes with $h^{-1}\tau h=:\delta_h$ for
    each $g,h\in \G$.  
  \item $\prod_{x\in F_l}\delta_{gx}\delta_{gt_nx}$ is trivial for
    each $n\in I$ and each $g\in \G$.
  \item If $\G=\integers\wr\integers$, in addition we need the relations of
    this group: $s_2^ns_1s_2^{-n}$ commutes with $s_1$ for each $n\in\integers$.
  \end{itemize}
  As it is easy to list all elements of $\G$, starting with the Turing machine
  for $I$ we can produce a Turing machine
  listing all these relations, i.e.~this presentation is recursive.

  Assume, on the other hand, that there is a Turing machine producing all the
  relations in $G_I$. In particular, it will list all the words
$ w_{t_n}= \prod_{x\in F_l}\delta_{x}\delta_{t_nx}$
  for the $n\in I$. Because the word problem in $\G$ and in
  $\integers_2[\G]$ is solvable, we can recognize these words and
  determine the $n\in I$ from them. On the other hand, by Lemma
  \ref{lem:normalform} if $b\notin I$ (i.e.~$t_b\notin\Lambda_I$) then
  $w_{t_b}\notin V_{F_l,I}$ i.e.~$w_{t_b}$ and therefore $b$ is not listed. In
  other words, this algorithm produces exactly the elements of $I$, and hence
  $I$ is recursively enumerable.  
\end{proof}

\begin{theorem}\label{theo:solvable_word}
  The group $G_I$ does have solvable word problem if and only if $I$ is
  recursive, i.e.~there is a Turing machine listing the elements of $I$ and
  another one listing those of the complement of $I$.
\end{theorem}
\begin{proof}
  Assume that $G_I$ has a solvable word problem. This means that, if we write
  down $w_{t_n}=\prod_{x\in F_l} \delta_x\delta_{t_ngx}$ we can decide whether
  $w_{t_n}=e$ or not, i.e.~$w_{t_n}\in V_{F_l,I}$ or not. By Lemma
  \ref{lem:normalform} this means that we can decide whether $t_n\in\Lambda_I$
  or not, i.e.~whether $n\in I$.

Let us now suppose that $I$ is recursive.
  There
  is a (computable) normal form for each element of $\integers_2^{\oplus\G}\semiprod\G$,
  written as the product of an element of $\integers_2^{\oplus\G}$
  and of an element of $\G$. It follows that, since $\G$ has solvable world
  problem, the word 
  problem in $G_I$ is solvable if and only if it is solvable in the
  normal subgroup $\integers_2^{\oplus\G}/V_{F_l,I}$. 

  This is equivalent to solving whether an element $x\in\integers_2^{\oplus\G}$
  belongs to $V_{F_l,I}$. This can be done as follows (provided $I$ is
  recursive):

 The function $x$ is finitely supported on $\G$ with values in
  $\integers_2$. Consider, as in the proof of Lemma \ref{lem:normalform}, its
  restriction to the coset $C:=g\langle s_1\rangle$. Assume this restriction
  is non-zero and form the (minimal) support interval
  $\{gs_1^{-1}, gs_1,\ldots,gs_1^d\}$ as in Lemma \ref{lem:normalform} (we
  choose $g$ in 
  its coset appropriately).  Because we can solve the word problem in $\G$, we
  can compute all these non-empty support intervals for the different cosets
  of $\langle s_1\rangle$.

By Lemma \ref{lem:normalform}
if  $d\le 0$ then $x\notin V_{F_l,I}$.

 Otherwise we now check, using that $I$ is recursive together
  with Lemma \ref{lem:relations}, whether there is $t\in\Lambda_I$
  such that $gt$ is in the interior of {another} support interval, by checking
  whether $g^{-1}g'\in\Lambda_I\langle e_1\rangle$ for the finitely many
  cosets $g'\langle s_1\rangle$ intersecting the support of $x$ non-trivially.

  If this is not the case, then by Lemma \ref{lem:normalform} $x\notin
  V_{F_l,I}$. Otherwise, subtract $gu-gtu$ (which is an element of
  $V_{F_l,I}$) from $x$ and continue as above. 

  This decreases the sum of the lengths of the support intervals. Therefore,
  after finitely many steps, either we observe that $x\notin V_{F_l,I}$ or the
 support is empty, i.e.~$x\in V_{F_l,I}$. 
\end{proof}

\section{Finitely presented groups}

\begin{theorem}\label{theo:main_two}
  There is an explicitly given finitely presented group $G$ and element $A\in
  \integers[G]$ 
  such that $\dim_{G}\ker(A)$ is transcendental.
  
  Consequently, there is a compact manifold $M$ such that an $L^2$-Betti number of the
  universal covering is transcendental.
\end{theorem}
\begin{proof}
  In Theorem \ref{theo:main} we give an explicit construction of $G$ and $A$
  such that $\dim_G\ker(A)=\beta_1+\beta_2\sum_{k=1}^\infty 2^{-dn_k+k}$ for
  every subset $I=\{n_1<n_2<\dots\}\subset\naturals$. 

  Moreover, if $I$ is recursively enumerable, e.g.~$I=\{k!\mid k\in\naturals\}$ then the
  corresponding group $G$ has a recursive presentation by Theorem
  \ref{theo:recursive_presentation}. If we use this set $I$, then the
  resulting $\sum_{k=1}^\infty 2^{-dk!+k}$ is transcendental as it is a
  Liouville number \cite{liouville44:_sur_des_class_tres_etend}.

  Finally, as explained in the introduction, we use e.g.~\cite[Proposition
3.1]{MR1828605} and replace $G$ by a finitely presented supergroup.
To produce such a finitely presented group which contains the recursively presented
groups $G$, we use Higman's theorem \cite{MR0130286}. How to explicitly construct the
supergroup and its presentation is shown nicely in \cite[Chapter 12, p.~450
ff]{MR1307623}. 
\end{proof}

\begin{remark}
  The finitely presented groups in Theorem \ref{theo:main_two} are obtained
  via application of Higman's embedding theorem. Unfortunately, although the
  recursively presented groups used as input for this theorem can be arranged
  to be solvable, this can not be expected for the resulting finitely
  presented group (indeed, the method of proof will produce groups which
  contain non-abelian free subgroups). Moreover, the construction in principle
  is explicit, but in practice the finite presentation obtained will be
  extremely cumbersome.

  Some of the examples of Grabowski
  \cite{grabowski10:_turin_machin_dynam_system_and_atiyah_probl} are much more
  explicit and give solvable (hence amenable) groups.
\end{remark}

\begin{remark}
  Using the method of proof of Proposition \ref{prop:set_closure} one can obtain
  many transcendental numbers which occur as $L^2$-Betti numbers of
  universal coverings of manifolds, or equivalently as kernel-dimensions for
  elements in the group ring of finitely presented groups. In particular, one
  can obtain all numbers of the form $\sum_{n\in I} 2^{-n}$ for a subset
  $I\subset\naturals$ which is recursive. In this case, moreover, we can
  arrange that the groups in question have a solvable word problem by Theorem
  \ref{theo:solvable_word}. 
\end{remark}

\begin{remark}
  Grabowski obtains all numbers $\sum_{n\in I} 2^{-n}$ where $I$ is
  recursively enumerable. We obtain all $\sum_{k=1}^\infty 2^{-dn_k+k}$ for
  $I=\{n_1<n_2<\dots\}$ recursively enumerable. Itai and Dror Bar-Natan explained to us
  that these two 
  classes of groups do not coincide\footnote{Let $J=\{n_1<n_2<\dots\}$ be an infinite
    recursively enumerable, but not 
    recursive set. Set $I:=\{2^{n_k}\}$ which is then also recursively
    enumerable. But $TI:=\{2^{n_k}+k\}$ is not recursively enumerable:
    otherwise, as $k\le n_k<< 2^{n_k}$ one could recover from the $2^{n_k}+k$
    also $k$ (and $n_k$). But the information that $n_k$ is the $k$-th
    smallest element of $J$ allows us, by waiting until $k-1$ smaller elements
    of $J$ are listed, to determine exactly the elements of $J$ which are
    smaller than $n_k$ and eventually to decide which numbers are in $J$ and
    which are not ---contradicting that $J$ is not recursive.}. Variations of
  the constructions will yield yet other values.

  It is clear that there are all together only countably many possible
  $L^2$-Betti numbers using the integral group ring of finitely presented
  groups (as the set of isomorphism classes of these groups is countable).

  It is an open question how this set exactly looks like. In
  \cite{Groth} it is implicitly discussed that for any
  $L^2$-Betti number $r$ of the universal covering of a finite CW-complex there
  is a Turing machine which produces a sequence of rational numbers whose
  limit superior is $r$. 

  In \cite{Groth}, it is also shown that an
  $L^2$-Betti number obtained from a finitely presented group with solvable
  word problem which is of $L^2$-determinant class (as introduced in
  \cite{MR1828605}, and satisfied for all the groups we constructed in this
  article) is of the form $\sum_{n\in I} 2^{-n}$ for a recursive
  subset $I\subset\naturals$. Consequently, these are precisely the
  $L^2$-Betti numbers obtained with groups which have a solvable word problem
  and satisfy the determinant conjecture.
\end{remark}



\begin{remark}
  The construction we have described here allows for many
  modifications. Essentially, we can make an operator $A$ which accepts
  \emph{local} patterns in the Cayley graph of $\G$. One interesting
  modification would be to only accept 1-neighborhoods of hooks with a
  thickened neighborhood of the ends. 

  Then one could replace in the definition of the quotient groups the set
  $F_l$ by a slightly larger set $F=\{e,s_1,s_1^2,s_1^3\}$. Its translates   only fit
  into the relevant set (the $1$-neighborhood of the hook with thickened ends)
  at the ends. This way one could arrange to have identifications in such
  subsets only if the two legs of the hook have equal length, and to have
  exactly one identification in this case. Nothing else changes, but the final
  sum corresponding to the calculation of Theorem \ref{theo:main}  gives
  \begin{equation*}
    \beta_1'+\beta_2' \sum_{k=1}^\infty 2^{-dn_k}.
  \end{equation*}
  It is then easy to see that, using recursively presented groups, we can get
  all numbers $\sum_{k=1}^\infty 2^{-n_k}$ with $I=\{n_1<n_2<\dots\}$
  recursively enumerable. Consequently, these numbers are also obtained as
  $L^2$-Betti numbers of universal coverings of compact manifolds.
\end{remark}

\begin{remark}
  We can go even one step further with our modifications and instead of hooks
  with two vertical legs work with hooks with left leg vertical as before, but
  right leg horizontal $\{g,gs_1,\dots,gs_1^d\}$. Again, one looks at the
  $1$-neighborhood of such hooks, but with thickened ends.

  Finally, one uses  $F=\{s_1^{-2},s_1^{-1},e,s_1,s_2,s_2^2,s_2^{-1}\}$ in the form of a cross.
  Then translates of $F$ fit only into our neighborhoods of the hook if they are
  placed in the end.

  Instead of the subgroups $\Lambda_I$ one works with subgroups $\Lambda_I'$
  generated by $s_1^ns_2^n$ for $n\in I$. At least if $\Gamma$ is free, this
  subgroup is free on these generators and has appropriate properties
  corresponding to those of $\Lambda_I$ we used above, and the extension lemma
  works for the cross $F$ (use the proof of \cite[Lemma 5.5]{austin09:_count_to_quest_of_atiyah}).

  We can then arrange the local patterns at the two ends (which are locally
  different: one is horizontal, the other vertical) to differ for those which
  contribute to the eigenvalue $-2$ in such a way that extension is \emph{not}
    possible. It follows that, instead of an identification which increases
    the weight of the contribution to the spectrum, those paths where both
    legs have length $n_k\in I$ do not contribute at all.

    Carrying out the calculations, we obtain for $I$ recursively enumerable an
    $L^2$-Betti number (for a 
    recursively presented group) of the form
    \begin{equation*}
      \beta_1''-\beta_2'' \sum_{k\in I}^\infty 2^{-dk};\quad\text{with $\beta_1'',\beta_2''\in\rationals$, $d\in\naturals$}.
    \end{equation*}

  We haven't checked, but expect, that the same works with
  $\G=\integers\wr\integers$. 
\end{remark}

\bibliographystyle{plain}
\bibliography{L2}

\begin{thebibliography}{10}

\bibitem{MR0420729}
M.~F. Atiyah.
\newblock Elliptic operators, discrete groups and von {N}eumann algebras.
\newblock In {\em Colloque ``{A}nalyse et {T}opologie'' en l'{H}onneur de
  {H}enri {C}artan ({O}rsay, 1974)}, pages 43--72. Ast\'erisque, No. 32--33.
  Soc. Math. France, Paris, 1976.

\bibitem{austin09:_count_to_quest_of_atiyah}
Tim Austin.
\newblock Rational group ring elements with kernels having irrational
  dimension.
\newblock {\em Proc. Lond. Math. Soc. (3)}, 107(6):1424--1448, 2013.

\bibitem{MR1934693}
Warren Dicks and Thomas Schick.
\newblock The spectral measure of certain elements of the complex group ring of
  a wreath product.
\newblock {\em Geom. Dedicata}, 93:121--137, 2002.

\bibitem{MR1990479}
J{\'o}zef Dodziuk, Peter Linnell, Varghese Mathai, Thomas Schick, and Stuart
  Yates.
\newblock Approximating {$L^2$}-invariants and the {A}tiyah conjecture.
\newblock {\em Comm. Pure Appl. Math.}, 56(7):839--873, 2003.
\newblock Dedicated to the memory of J{\"u}rgen K. Moser.

\bibitem{grabowski10:_turin_machin_dynam_system_and_atiyah_probl}
{\L}ukasz Grabowski.
\newblock On {T}uring dynamical systems and the {A}tiyah problem.
\newblock {\em Invent. Math.}, 198(1):27--69, 2014.

\bibitem{MR1797748}
Rostislav~I. Grigorchuk, Peter Linnell, Thomas Schick, and Andrzej {\.Z}uk.
\newblock On a question of {A}tiyah.
\newblock {\em C. R. Acad. Sci. Paris S\'er. I Math.}, 331(9):663--668, 2000.

\bibitem{MR1866850}
Rostislav~I. Grigorchuk and Andrzej {\.Z}uk.
\newblock The lamplighter group as a group generated by a 2-state automaton,
  and its spectrum.
\newblock {\em Geom. Dedicata}, 87(1-3):209--244, 2001.

\bibitem{Groth}
Thorsten Groth.
\newblock $l^2$-{B}ettizahlen endlich pr\"asentierter {G}ruppen.
\newblock Bachelorarbeit, Georg-August-Universit\"at G\"ottingen, 2012.

\bibitem{MR0130286}
G.~Higman.
\newblock Subgroups of finitely presented groups.
\newblock {\em Proc. Roy. Soc. Ser. A}, 262:455--475, 1961.

\bibitem{lehner:_free_lampl_group_and_quest_of_atiyah}
Franz Lehner and Stephan Wagner.
\newblock Free lamplighter groups and a question of {A}tiyah.
\newblock {\em Amer. J. Math.}, 135(3):835--849, 2013.

\bibitem{MR1242889}
Peter~A. Linnell.
\newblock Division rings and group von {N}eumann algebras.
\newblock {\em Forum Math.}, 5(6):561--576, 1993.

\bibitem{liouville44:_sur_des_class_tres_etend}
Joseph Liouville.
\newblock Sur des classes tr\`es \'etendues de quantit\'es dont la valeur n'est
  ni alg\'ebrique, ni m\^eme r\'eductible \`a des irrationelles alg\'ebriques.
\newblock {\em C. R. Acad. Sci. Paris}, 18:883–885 and 993–995, 1844.

\bibitem{MR1926649}
Wolfgang L{\"u}ck.
\newblock {\em {$L^2$}-invariants: theory and applications to geometry and
  {$K$}-theory}, volume~44 of {\em Ergebnisse der Mathematik und ihrer
  Grenzgebiete. 3. Folge.}
\newblock Springer-Verlag, Berlin, 2002.

\bibitem{MR1307623}
Joseph~J. Rotman.
\newblock {\em An introduction to the theory of groups}, volume 148 of {\em
  Graduate Texts in Mathematics}.
\newblock Springer-Verlag, New York, fourth edition, 1995.

\bibitem{MR1828605}
Thomas Schick.
\newblock {$L^2$}-determinant class and approximation of {$L^2$}-{B}etti
  numbers.
\newblock {\em Trans. Amer. Math. Soc.}, 353(8):3247--3265 (electronic), 2001.

\end{thebibliography}

\end{document}